\documentclass{article}
\usepackage[utf8]{inputenc}
\usepackage[hidelinks]{hyperref}

\usepackage{amsmath, amsfonts, amssymb, amsthm}
\usepackage{tikz-cd}
\usepackage{stringdiagrams}
\usepackage{bussproofs}
\usepackage{eqproof}

\title{Generalizations of Bilinear Maps\\ {\small Technical Report}}
\author{Tom\'a\v s Jakl and Dan Marsden and Nihil Shah}

\newcommand{\df}[1]{\emph{#1}}
\newcommand{\nt}{\Rightarrow}

\newcommand{\st}{\mathsf{st}}
\newcommand{\stp}{\mathsf{st}'}
\newcommand{\dst}{\mathsf{dst}}
\newcommand{\dstp}{\mathsf{dst}'}

\newcommand{\monad}[1]{\mathbb{#1}}
\newcommand{\MT}{\monad{T}}
\newcommand{\MS}{\monad{S}}
\newcommand{\MC}{\monad{C}}
\newcommand{\MD}{\monad{D}}
\newcommand{\ML}{\monad{L}}
\newcommand{\MM}{\monad{M}}
\newcommand{\category}[1]{\mathcal{#1}}
\newcommand{\CV}{\category{V}}
\newcommand{\CC}{\category{C}}
\newcommand{\CD}{\category{D}}
\newcommand{\CSet}{\mathsf{Set}}

\newcommand{\FT}{F}

\newcommand{\alg}[1]{\mathsf{Alg}(#1)}
\newcommand{\coalg}[1]{\mathsf{CoAlg}(#1)}

\newcommand{\klc}[2]{#2_{#1}}
\newcommand{\emc}[2]{#2^{#1}}
    
    \newcommand{\coemc}[2]{#2^{#1}}

\newcommand{\klf}[2]{#1_{#2}}
\newcommand{\emf}[2]{#1^{#2}}
\newcommand{\emklf}[2]{\widehat{#1_{#2}}} 

\newcommand{\klfree}[1]{{F_{#1}}}
\newcommand{\eilmofree}[1]{F^{#1}}
    \newcommand\klfr{\klfree}
    \newcommand\emfr{\eilmofree}

\newcommand{\eilmoforget}[1]{U^{#1}}
\newcommand{\kleiforget}[1]{U_{#1}}
    \newcommand\emfg{\eilmoforget}
    \newcommand\klfg{\kleiforget}

\newcommand{\klar}[1]{\rightarrow_{#1}}
\newcommand{\emar}[1]{\rightarrow^{#1}}

\newcommand{\id}{\mathsf{id}}
\newcommand{\Id}{\mathsf{Id}}

\usepackage{amsthm}
\theoremstyle{plain}
\newtheorem{theorem}{Theorem}

\newtheorem{proposition}[theorem]{Proposition}
\theoremstyle{definition}
\newtheorem{definition}[theorem]{Definition}
\newtheorem{example}[theorem]{Example}

\theoremstyle{remark}
\newtheorem{remark}[theorem]{Remark}
\numberwithin{theorem}{section}

\begin{document}

\maketitle

\begin{abstract}
Bilinear maps and their classifying tensor products are well-known in the theory
of linear algebra, and their generalization to algebras of commutative
monads is a classical result of monad theory. Motivated by constructions needed in categorical
approaches to finite model theory, we generalize the notion of bimorphism much further.
To illustrate these maps are mathematically natural notions, we show that many common axioms in category
theory can be phrased as certain morphisms being bimorphisms.
We also show that much of the established theory of bimorphisms goes through in much greater generality. 
Our results carefully identify which assumptions are needed for the different components of the theory,
including when good properties hold globally, or can at least be established locally.
We include a brief string diagrammatic account of the bimorphism axiom, and conclude by recovering a simple
proof of a classical theorem, emphasizing the efficacy of the bimorphism perspective.
\end{abstract}

\tableofcontents

\section{Introduction}
In linear algebra, the operation of tensor product $\otimes$ is characterised by the universal property that bilinear maps of type $U \times V \rightarrow W$, involving vector spaces $U,V$ and $W$, are in bijective correspondence with linear maps of type $U \otimes V \rightarrow W$. (See for example~\cite{aluffi2021algebra} for an expository account in categorical language. )

The tensor product in vector spaces can been seen as a lifting of the cartesian product over $\CSet$ to algebras over the free vector space monad.
Under mild assumptions, this situation generalizes to any commutative monad $\MT$ by defining a lifting of a monoidal product on a category $\CC$ to a monoidal product on the category of algebras $\emc{\MT}{\CC}$.
The definition of commutative monad involves a lot of assumptions, but the notion of bilinear map, and the corresponding classifying object construction make no essential use of many of them.


In this paper, we identify a minimal notion of bimorphism, generalizing the classical theory of bilinear maps in the setting of monads~\cite{kock1971bilinearity} (see also \cite{jacobs1994semantics,seal2012tensors}). At this greater level of generality, bimorphisms subsume ordinary algebra morphisms. Furthermore, many axioms of classical monad theory, such as those for being a monad, or an Eilenberg-Moore algebra can be phrased as certain maps being bimorphisms. We establish that classifying objects for bimorphisms exist under very mild assumptions, and carefully identify the axioms required for good properties of the classifying object construction, such as functoriality or commuting with the free algebra functors.

Though the lifting of monoidal products to the category of algebras has been studied extensively, one of our main technical contribution is in extracting the minimal conditions under which various liftings of functors to endofunctor algebras exist. 
Another contribution is demonstrating, that in practice, a global assumption that all diagrams of a certain type commute
can be weakened to a local assumption, that the diagrams at a particular index commute.
For example, existence of a natural transformation can be weakened to commutativity of specific naturality squares. 
We also provide a string diagrammatic proof of a well-known adjoint lifting theorem~\cite{johnstone1975adjoint, keigher1975adjunctions} which utilises the theory of bimorphisms.   
Underlying this entire presentation, is the key unifying idea of a $\lambda$-bimorphism between endofunctor algebras.

Our main motivation for investigating bimorphisms is to study Feferman-Vaught-Mostowski (FVM) theorems~\cite{mostowski1952direct, feferman1967first}. Such theorems characterize how logical equivalence behaves under composition and transformation of models. They have applications ranging from model theory~\cite{gurevich1985monadic} to algorithmic meta-theorems~\cite{courcelle1990monadic, courcelle2000linear, courcelle2012graph, oum2006approximating, makowsky2004algorithmic}.

An FVM theorem typically states that, for a logic $\mathcal L$, and~$n$-ary operation~$H$ on models, the $\mathcal{L}$-theory of $H(A_1,\dots,A_n)$ can be computed from the $\mathcal{L}$-theories of $A_1,\dots,A_n$.
In our forthcoming paper~\cite{jaklmarsdenshah2022fvm} we make use of the technology built here, and demonstrate that in the setting of game comonads~\cite{AbramskyDW17, abramsky2021relating} the core of Feferman-Vaught-Mostowski theorems is lifting~$H$ to the Eilenberg-Moore category of a comonad capturing $\mathcal L$, and analysing  properties of bimorphisms. 
Our emphasis on identifying the minimal technical considerations is motivated by practical considerations in that work.
We focus on monads in the present work as the classical motivating example is algebraic in nature.

\section{Preliminaries}
We assume familarity of the standard category-theoretic notions of functor, natural transformations, adjunctions, and (co)monads. 
These notions can be found in any introductory writing on the subject (see e.g.\ \cite{mac2013categories})
For an endofunctor $T:\CC \rightarrow \CC$, the category of $T$-algebras is denoted $\alg{T}$ and the category of $T$-coalgebras is denoted $\coalg{T}$. 
For a monad $\MT = (T,\mu^\MT,\eta^\MT)$ over $\CC$, the \textit{Eilenberg-Moore category of $\MT$} is denoted $\emc{\MT}{\CC}$. 
The \textit{Kleisli category of $\MT$} is denoted $\klc{\MT}{\CC}$.
Recall that $\klc{\MT}{\CC}$ and $\emc{\MT}{\CC}$ are the initial and terminal object in the category of adjunctions which yield $\MT$ as monad over $\CC$.
In particular, there is a comparison functor $K:\klc{\MT}{\CC} \rightarrow \emc{\MT}{\CC}$ which maps the Kleisli adjunction $\klfg{\MT} \dashv \klfr{\MT} $ to the Eilenberg-Moore adjunction $\emfr{\MT} \dashv \emfg{\MT}$ associated to $\MT$ as described in figure~\ref{diag:mt-comparison}. 
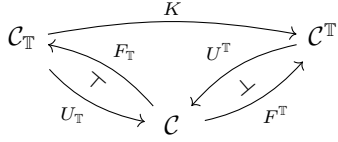
\begin{figure}
    \centering
    \begin{tikzcd}[column sep=4em]
        \klc{\MT}{\CC}
            \ar[bend right=18,swap, shorten <= 0.5em, shorten >= 0.3em]{dr}{\klfg{\MT}}
            \ar[bend left=10]{rr}{K}
        &
        & \emc{\MT}{\CC} 
            \ar[bend right=18,swap]{dl}{\emfg{\MT}}
        \\
        & \CC
            \ar[bend right=18,swap]{ul}{\klfr{\MT}}
            \ar[phantom,sloped,pos=0.55]{ul}{\rotatebox{80}{$\dashv$}}
            \ar[bend right=18,swap, shorten <= 0.5em, shorten >= 0.3em]{ur}{\emfr{\MT}}
            \ar[phantom,sloped]{ur}{\rotatebox{-80}{$\dashv$}}
    \end{tikzcd}

    \caption{The Kleisli and Eilenberg-Moore adjunctions for a monad}
    \label{diag:mt-comparison}
\end{figure}

In an effort to declutter notation, with natural transformations $\lambda:T \nt S$, we will often write $\lambda$ for the individual components $\lambda_A : T(A) \rightarrow S(A)$ when it does not lead to confusion.  
We will also write $\mu$ for $\mu^{\MT}$ when it does not lead to confusion and similarly for $\eta$. 
In section~\ref{sec:comonads}, we will work with comonads $\MD = (D,\delta^\MD,\epsilon^\MD)$ where will employ the same terminological and notational conventions.

\section{Bilinear maps}
For real vector spaces~$U,V,W$, a standard algebraic notion is that of a bilinear map. We say
that a function~$h : U \times V \rightarrow W$ is \emph{bilinear} if:
\begin{enumerate}
    \item For all~$u \in U$, $h(u,-)$ is a linear map~$V \rightarrow W$.
    \item For all~$v \in V$, $h(-,v)$ is a linear map~$U \rightarrow W$.
\end{enumerate}
Furthermore, there is a vector space~$U \otimes V$, the \df{tensor product} of~$U$ and~$V$, such that there is a bijective correspondence between:
\begin{enumerate}
    \item Bilinear maps~$U \times V \rightarrow W$.
    \item Linear maps~$U \otimes V \rightarrow W$.
\end{enumerate}
This universal property of the tensor product is very convenient, as it allows us to move freely between the perspective of bilinear maps, and that of linear maps, for which there is a great deal of linear algebraic machinery.

It is a classical fact of monad theory that this situation generalizes to many other settings~\cite{kock1971bilinearity}, particularly to that of commutative monads (See also~\cite{jacobs1994semantics, seal2012tensors}). In this paper, we explore this underappreciated fact in some detail. We establish connections with so-called Kleisli laws, and use this to show that the notion of bilinear map can be generalized yet further, with a corresponding universal object. We show that the abstract formulation of bilinear maps captures a wide range of familiar notions from mathematics and computer science.

\section{Generalizing Bilinear maps to commutative monads}
This section discusses some classical categorical results relating to bilinear maps.
We must first introduce the notions of strong and commutative monads. The material in this section is standard, but provides motivation for later constructions, and gives us a convenient opportunity to fix some terminology and notation.

\subsection{Strong and Commutative Monads}
Let~$(\mathcal{V}, \otimes, I)$ be a symmetric monoidal category, with coherence isomorphisms:
\newcommand{\leftunitor}{\mathsf{l}}
\newcommand{\rightunitor}{\mathsf{r}}
\newcommand{\associator}{\mathsf{a}}
\newcommand{\symmetry}{\mathsf{s}}
\begin{align*}
    \leftunitor : I \otimes A &\rightarrow A \\
    \rightunitor : A \otimes I &\rightarrow A \\
    \associator : (A \otimes B) \otimes C &\rightarrow A \otimes (B \otimes C)\\
    \symmetry : A \otimes B &\rightarrow B \otimes A
\end{align*}
(For background on symmetric monoidal categories, see for example~\cite{mac2013categories}).

A \df{strength} for an endofunctor~$\FT : \mathcal{V} \rightarrow \mathcal{V}$ is a natural transformation:
\[
    \st : A \otimes \FT(B) \rightarrow \FT(A \otimes B)
\]
such that the following two diagrams commute:
\[
\begin{tikzcd}[column sep=0.5em]
    I \otimes \FT(A) \ar{rr}{\st} \drar[swap]{\leftunitor} & & \FT(I \otimes A) \ar{dl}{F(\leftunitor)} \\
    & \FT(A)
\end{tikzcd}
\]
\[
\begin{tikzcd}
(A \otimes B) \otimes \FT(C) \arrow[rr, "\st"] \dar[swap]{\associator} & & \FT((A \otimes B) \otimes C) \dar{\FT(\associator)} \\
A \otimes (B \otimes \FT(C)) \rar[swap]{A \otimes \st} & A \otimes \FT(B \otimes C) \rar[swap]{\st} & \FT(A \otimes (B \otimes C))
\end{tikzcd}  
\]

Dually, a \df{costrength} for~$\FT$ is a natural transformation:
\[
    \stp : \FT(A) \otimes B \rightarrow \FT(A \otimes B)
\]
satisfying the dual of the conditions above. If~$\CV$ is a symmetric monoidal category, we can define a costrength from a strength using the symmetry~$\symmetry$ as the composite:
\[
\begin{tikzcd}
\FT(A) \otimes B \rar{\symmetry} & B \otimes \FT(A) \rar{\st} & \FT(B \otimes A) \rar{\FT(\symmetry)} & \FT(A \otimes B) 
\end{tikzcd}
\]
A \df{strong functor} is a functor with a specified strength.

A monad~$(\MT,\eta,\mu)$ is~\df{strong}~\cite{Kock1970} if $\MT$ is a strong functor, such that the strength is compatible with the monad structure, in that the following diagrams commute:
\[
\begin{tikzcd}[column sep=0.5em]
    A \otimes \MT(B) \ar{rr}{\st} & & \MT(A \otimes B) \\
    & A \otimes B \ar{ul}{A \otimes \eta} \ar[swap]{ur}{\eta}
\end{tikzcd}
\]
\[
\begin{tikzcd}
A \otimes \MT^2(B) \dar[swap]{\st} \arrow[rr, "A \otimes \mu"] & & A \otimes \MT(B) \dar{\st} \\
\MT(A \otimes \MT(B)) \rar[swap]{\MT(\st)} & \MT^2(A \otimes B) \rar[swap]{\mu} & \MT(A \otimes B)
\end{tikzcd}
\]
\begin{remark}
On a monoidal closed category, a strong monad is the same thing as an enriched monad~\cite{Kock1972}.
Every~$\CSet$ monad is~$\CSet$-enriched, and hence canonically strong, with $\st(a, t) := \MT(\lambda b. (a,b))(t)$.
\end{remark}
For a strong monad on a symmetric monoidal category, we define 
\[ \dst : \MT(A) \otimes \MT(B) \rightarrow \MT(A \otimes B)  ,\] using the induced costrength, as the composite:
\[
\begin{tikzcd}
\MT(A) \otimes \MT(B) \rar{\st} & \MT(\MT(A) \otimes B) \rar{\MT(\stp)} & \MT^2(A \otimes B) \rar{\mu} & \MT(A \otimes B)
\end{tikzcd}
\]
We can also define natural transformation 
\[ \dstp : \MT(A) \otimes \MT(B) \rightarrow \MT(A \otimes B) \]
as the composite:
\[
\begin{tikzcd}
\MT(A) \otimes \MT(B) \rar{\stp} & \MT(\MT(A) \otimes B) \rar{\MT(\st)} & \MT^2(A \otimes B) \rar{\mu} & \MT(A \otimes B)
\end{tikzcd}
\]
The morphisms~$\dst$ and~$\dstp$ are referred to as \df{double strengths}. A strong monad is said to be~\df{commutative}~\cite{Kock1970} if~$\dst = \dstp$.
\begin{remark}
    On a symmetric monoidal closed category, a commutative monad is the same thing as a symmetric monoidal monad~\cite{Kock1972}.
\end{remark}
\begin{example}
For the list monad~$\ML$, we have:
\begin{align*}
    \st(a,[b_1,\ldots,b_n]) &= [(a,b_1),\ldots,(a,b_n)] \\
    \stp([a_1,\ldots,a_n], b) &= [(a_1,b),\ldots,(a_n,b)] \\
    \dst([a_1,\ldots,a_m],[b_1,\ldots,b_m]) &= [(a_1,b_1),\ldots,(a_m,b_1),\ldots,(a_1,b_n),\ldots, (a_m,b_n)]\\
    \dstp([a_1,\ldots,a_m],[b_1,\ldots,b_m]) &= [(a_1,b_1),\ldots,(a_1,b_n),\ldots,(a_m,b_1),\ldots,(a_m,b_n)]
\end{align*}
In particular, the list monad is~\emph{not} commutative.
\end{example}

\begin{example}
If~$\MT$ is a monad induced by the presentation of some algebraic theory~$(\Sigma,E)$, $\st(a,t)$ is defined on the representative term~$t$ by replacing the occurrence of each variable~$b$ with~$(a,b)$. Dually, $\stp(t,b)$ replaces the occurrence of each variable~$a$ in~$t$ with~$(a,b)$. 

For example, if~$\Sigma = \{\times,1\}$, and the equations~$E$ are those for the theory of monoids, then we have:
\begin{align*}
    \st(a, (b_1, \times b_2) \times b_3) &= ((a,b_1) \times (a,b_2)) \times (a,b_3) \\
    \stp((a_1 \times a_2) \times a_3, b) &= ((a_1,b) \times (a_2,b)) \times (a_3,b)
\end{align*}
As we would expect, this is a syntactic rephrasing of the actions for the list monad.
\end{example}

\begin{example}
For a semiring~$S$, we can define a monad~$\MM_S : \CSet \rightarrow \CSet$:
\begin{description}
\item[Endofunctor]: $\MM_S(A)$ is the set of finite formal sums of the form:
\[
\sum_i s_i a_i
\]
Here each~$s_i$ is an element of the semiring, and each~$a_i$ is in~$A$.
\item[Unit]: The unit maps~$a$ to the trivial sum.
\item[Multiplication]: The multiplication flattens a sum of sums to a sum:
\[
\sum_i s_i (\sum_{j_i} s_{i,j} a_{i,j}) \mapsto \sum_{i,j} (s_i \times s_{i,j}) a_{i,j}
\]
\end{description}
The strength and costrength are given by:
\begin{align*}
    \st(a, \sum_j s_j b_j) &= \sum_j s_j (a,b_j) \\
    \stp(\sum_i s_i a_i, b) &= \sum_i s_i (a_i,b)
\end{align*}
Therefore we have
\begin{align*}
    \dst(\sum_i s_i a_i, \sum_j r_j b_j) = \sum_{i,j} (r_j \times s_i)(a_i,b_j) \\
    \dstp(\sum_i s_i a_i, \sum_j r_j b_j) = \sum_{i,j} (s_i \times r_j)(a_i,b_j)
\end{align*}
$\MM_S$ will be a commutative monad if and only if the semiring multiplication is commutative. The algebras for this monad are known as semimodules, and modules if~$S$ is in fact a ring. We note a few special cases of interest, in each case, the corresponding monad is commutative:
\begin{itemize}
    \item If~$S = \mathbb{N}$, the algebras are abelian monoids.
    \item If~$S = \mathbb{Z}$, the algebras are abelian groups.
    \item If~$S$ is a field~$F$, the algebras are the vector spaces over~$F$.
\end{itemize}

\end{example}

\subsection{Abstracting Bilinear Maps}

For a commutative monad~$\MT$, and algebras~$(A,\alpha)$, $(B,\beta)$ and~$(C,\gamma)$, following~\cite{kock1971bilinearity}, we say that a morphism~$h : A \otimes B \rightarrow C$ is \df{bilinear}, or a~\df{bimorphism}, if the following diagram commutes:

\begin{equation}
\label{eq:commutative-monad-bimorphism-diagram}
\begin{tikzcd}
\MT(A) \otimes \MT(B) \rar{\dst} \dar[swap]{\alpha \otimes \beta} & \MT(A \otimes B) \rar{\MT(h)} & \MT(C) \dar{\gamma} \\
A \otimes B \arrow[rr, swap, "h"] & & C
\end{tikzcd}
\end{equation}

\begin{example}
If~$\MT$ is the abelian monoid monad, a function~$h : A \times B \rightarrow C$ is bilinear if it satisfies:
\[
h(\sum_i a_i, \sum_j b_j) = \sum_i \sum_j h(a_i, b_j) 
\]
This is equivalent to it being a homomorphism component-wise. That is, satisfying the equations
\[
    h(\sum_i a_i, b) = \sum_i h(a_i,b)
    \qquad\mbox{and}\qquad
    h(a, \sum_j b_j) = \sum_j h(a, b_j) .
\]
\end{example}

Bilinearity can also be phrased in terms of the left and right components separately,
using the monad strength and costrength.
\begin{equation}
\label{eq:commutative-monad-bimorphism-right-component}
\begin{tikzcd}
A \otimes \MT(B) \rar{\st} \dar[swap]{A \otimes \beta} & \MT(A \otimes B) \rar{\MT(h)} & \MT(C) \dar{\gamma} \\
A \otimes B \arrow[rr, swap, "h"] & & C
\end{tikzcd}
\end{equation}

\begin{equation}
\label{eq:commutative-monad-bimorphism-left-component}
\begin{tikzcd}
\MT(A) \otimes \MT(B) \rar{\stp} \dar[swap]{\alpha \otimes B} & \MT(A \otimes B) \rar{\MT(h)} & \MT(C) \dar{\gamma} \\
A \otimes B \arrow[rr, swap, "h"] & & C
\end{tikzcd}
\end{equation}
 A morphism~$h : A \otimes B \rightarrow C$ is bilinear in the sense that diagram~\eqref{eq:commutative-monad-bimorphism-diagram} commutes if and only if both diagrams~\eqref{eq:commutative-monad-bimorphism-right-component} and~\eqref{eq:commutative-monad-bimorphism-left-component} commute.

\begin{example}
If~$\MT$ is the monoid monad, that diagrams~\eqref{eq:commutative-monad-bimorphism-right-component} and~\eqref{eq:commutative-monad-bimorphism-left-component} commute is equivalent to equations:
\[
    h(\prod_i a_i, b) = \prod_i h(a_i,b)
    \qquad\mbox{and}\qquad
    h(a, \prod_j b_j) = \prod_j h(a, b_j) .
\]
Here, crucially, the notation intends that we preserve the order of the elements in the products.

As we know the monoid monad is not commutative, we would expect that the conditions implied by~$\dst$ and~$\dstp$ are different. This is the case, with the two properties being
\[
    h(\prod_i a_i, \prod_j b_j) = \prod_j \prod_i h(a_i,b_j)
    \qquad\mbox{and}\qquad
    h(\prod_i a_i, \prod_j b_j) = \prod_i \prod_j h(a_i,b_j) .
\]
Again, it is critical we preserve the order of the elements in the products. It is perhaps easier to see the difference on smaller examples of each condition:
\begin{align*}
    h(a \times a', b \times b') &= h(b,a) \times h(b,a') \times h(b',a) \times h(a',b') \\
    h(a \times a', b \times b') &= h(b,a) \times h(b',a) \times h(b,a') \times h(a',b')
\end{align*}
The two conditions are not equivalent, as the monoid multiplication is not in general commutative.
\end{example}

\subsection{Lifting tensor products of commutative monads}
\label{sec:lifting-tensor-products}
The following is a well-known and very useful result, see for example~\cite{jacobs1994semantics}.

\begin{theorem}
    \label{thm:lifting-smc-structure}
    If~$\MT$ is a commutative monad on a symmetric monoidal category~$(\CV,\otimes,I)$ then:
\begin{enumerate}
    \item The symmetric monoidal structure lifts to a symmetric monoidal structure $(\otimes_\MT,I_\MT)$ on~$\klc{\MT}{\CV}$.
    \item If~$\emc{\MT}{\CV}$ has coequalizers of reflexive pairs, then the symmetric monoidal structure lifts to a symmetric monoidal structure $(\otimes^\MT,I^\MT)$ on~$\emc{\MT}{\CV}$, and the full and faithful functor~$K : \klc{\MT}{\CV} \rightarrow \emc{\MT}{\CV}$ preserves the symmetric monoidal structure (in sense of \cite{jacobs1994semantics}). 
    \item Furthermore, $\otimes^\MT$ is universal, in the sense that for algebras~$(A,\alpha)$ and $(B,\beta)$ there is a bimorphism~$u : A \otimes B \rightarrow (\alpha \otimes^\MT \beta)$ such that for every coalgebra~$(C,\gamma)$, every bimorphism~$h : A \otimes B \rightarrow C$ factors through a unique~$\hat{h}$ as:
    \[
    \begin{tikzcd}
        A \otimes B \drar[swap]{h} \rar{u} & (\alpha \otimes^\MT \beta) \dar{\hat{h}} \\
        & C
    \end{tikzcd}
    \]
\end{enumerate}
\end{theorem}
Here we see the connection between bilinear maps and universal tensors extended to the setting of commutative monads on some monoidal category. We shall move beyond this setting in later sections.

\section{Beyond Commutative Monads}
As far as we are aware, the existing literature on bimorphisms, and their corresponding classifying objects in Eilenberg-Moore categories, restricts attention to commutative monads on monoidal categories. We now pursue a wide generalization of these notions. Throughout, we will be careful to identify the roles of the various assumptions play in the theory, disentangling the large number of assumptions that are typically combined in the commutative monad setting.

\subsection{Bimorphisms}
We observe that diagram~\eqref{eq:commutative-monad-bimorphism-diagram} makes no essential use of the following features:
\begin{enumerate}
    \item That~$\MT$ is a monad rather than a mere endofunctor.
    \item That the algebras on the left and right of the diagram are both~$\MT$-algebras. 
    \item The monoidal structure.
    \item That~$\otimes$ is a bifunctor, i.e.\ that it is 2-ary.
    \item That~$\dst$ is natural and satisfies various coherence axioms.
\end{enumerate}
Abstracting away from these inessential details, we arrive at the following.
\begin{definition}[Left $\lambda$-morphism for endofunctor-algebras]
\label{def:left-endofunctor-bimorphism}
Let $S : \CC \rightarrow \CC$ and $T : \CD \rightarrow \CD$ be endofunctors, and $H : \CC \rightarrow \CD$ a functor.
For a morphism $\lambda : H(S(A)) \rightarrow T(H(A))$, $S$-algebra $(A,\alpha)$ and $T$-algebra~$(B,\beta)$, a morphism $h : H(A) \rightarrow B$ is a
\df{(left) $\lambda$-morphism} from $\alpha$ to $\beta$ if the following diagram commutes:
\begin{equation}
\label{eq:left-bimorphism}
\begin{tikzcd}
H(S(A)) \rar{\lambda} \dar[swap]{H(\alpha)} & T(H(A)) \rar{T(h)} & T(B) \dar{\beta} \\
H(A) \arrow[rr, swap, "h"] & & B
\end{tikzcd}
\end{equation}
In which case we shall write~$h : \alpha \klar{\lambda} \beta$. We shall also say~$h$ is a left bimorphism when we do not wish to specify~$\lambda$, or it is clear from the context.
\end{definition}
We shall also have need of a related condition.
\begin{definition}[Right $\rho$-morphism for endofunctor-algebras]
\label{def:right-endofunctor-bimorphism}
Let $S : \CC \rightarrow \CC$ and $T : \CD \rightarrow \CD$ be endofunctors, and $G : \CD \rightarrow \CC$ a functor.
For a morphism~$\rho : S(G(B)) \rightarrow G(T(B))$, $S$-algebra~$(A,\alpha)$ and~$T$-algebra~$(B,\beta)$, a morphism~$h : A \rightarrow G(B)$ is a
\df{(right) $\rho$-morphism} from $\alpha$ to $\beta$ if the following diagram commutes:
\begin{equation}
\label{eq:right-bimorphism}
\begin{tikzcd}
S(A) \rar{S(h)} \dar[swap]{\alpha} & S(G(B)) \rar{\rho} & G(T(B)) \dar{G(\beta)} \\
A \arrow[rr, swap, "h"] & & G(B)
\end{tikzcd}
\end{equation}
In which case we shall write~$h : \alpha \emar{\rho} \beta$.  We shall also say~$h$ is a right bimorphism when we do not wish to specify~$\rho$, or it is clear from the context.
\end{definition}
\begin{remark}
    The terms left and right correspond to the position of~$\lambda$ or~$\rho$ in diagrams~\eqref{eq:left-bimorphism} and~\eqref{eq:right-bimorphism}.
\end{remark}
We shall refer to morphisms satisfying either equation~\eqref{eq:left-bimorphism} or~\eqref{eq:right-bimorphism} as bimorphisms. This terminology follows~\cite{jacobs1994semantics}, and the extension of bilinearity terminology from the original~\cite{kock1971bilinearity}, although we push the notion even further. Admittedly, there is nothing intrinsically binary about these morphisms, but the terminology is concise, and maintains the connection with historical conventions.
\begin{remark}
Note that conditions~\eqref{eq:left-bimorphism} and~\eqref{eq:right-bimorphism} are not dual, although both capture being an algebra morphism ``up-to a mediating morphism''. As we shall see later the theory develops rather differently for the two, with condition~\eqref{eq:left-bimorphism} perhaps being the more interesting concept.
\end{remark}
Although we will generally start imposing stronger conditions on~$\lambda$ in the notion of bimorphism, for example naturality or compatibility with other structures, it is useful to identify this level of abstraction. This will allow us to correctly isolate which conditions are essential for later results.

\begin{example}[Endofunctor algebra morphisms]
$T$-algebra morphisms are a special case of bimorphisms, both in the sense of definition~\ref{def:left-endofunctor-bimorphism} and~\ref{def:right-endofunctor-bimorphism}. For $T$-algebras $(A,\alpha)$ and $(B,\beta)$, a morphism $h : A \rightarrow B$ is an algebra morphism if and only if:
\[
\begin{tikzcd}
    \Id(T(A)) \rar{\id} \dar[swap]{\Id(\alpha)} & T(\Id(A)) \rar{T(h)} & T(B) \dar{\beta} \\
    \Id(A) \arrow[rr, swap, "h"] & & B
\end{tikzcd}
\]
That is, if it is a~left ${\id}$-morphism. We also have that~$h$ is an algebra morphism if and only if it is a~right ${\id}$-morphism.
\end{example}
The following example shows that by generalizing to endofunctor algebras, we can state well-known
axioms as requiring certain morphisms be bimorphisms.
\begin{example}[Eilenberg-Moore axioms as bimorphisms]
For a monad~$\MT$, let~$\alpha : \MT(A) \rightarrow A$ be an \emph{endofunctor} algebra
for~$\MT$. That~$\id$ is a bimorphism of type~$\id \klar{\eta} \alpha$ is equivalent to requiring the following diagram to commute:
\[
\begin{tikzcd}
\Id(\Id(A)) \rar{\eta} \dar[swap]{\Id(\id)} & \MT(\Id(A)) \rar{\MT(\id)} & \MT(A) \dar{\alpha} \\
\Id(A) \arrow[rr, swap, "\id"] & & A
\end{tikzcd}
\]
Some trivial simplifications show this is the same as~$\alpha$ satisfying the unit axiom
for an Eilenberg-Moore algebra.

Similarly, that~$\alpha$ is a bimorphism of type~$\alpha \klar{\mu} \id$ is equivalent
to requiring the following diagram to commute:
\[
\begin{tikzcd}
\MT^2(A) \rar{\mu} \dar[swap]{\MT(\alpha)} & \Id(\MT(A)) \rar{\Id(\alpha)} & \Id(A) \dar{\id} \\
\MT(A) \arrow[rr, swap, "\alpha"] & & A
\end{tikzcd}
\]
Simplifying, this is exactly that~$\alpha$ satisfies the multiplication axiom for an
Eilenberg-Moore algebra.
\end{example}
As a second example of the ubiquity of bimorphism conditions, we note that the monad axioms themselves
can be phrased as bimorphism conditions.
\begin{example}[Monad axioms]
The unit axioms for a monad require that the following diagram commutes for all~$A$, both in the
case where~$\lambda_A = \eta_{\MT(A)}$ and~$\lambda_A = \MT(\eta_A)$:
\[
\begin{tikzcd}
\MT(\Id(A)) \rar{\lambda} \dar[swap]{\MT(\id)} & \MT(\MT(A)) \rar{\MT(\id)} & \MT^2(A) \dar{\mu} \\
A \arrow[rr, swap, "\id"] & & A
\end{tikzcd}
\]
That is~$\id_A$ is both:
\begin{enumerate}
    \item A left ${\eta_{\MT(A)}}$-morphism of type~$\id \klar{\eta_{\MT(A)}} \mu$.
    \item A left ${\MT(\eta_A)}$-morphism of type~$\id \klar{\MT(\eta_{A})} \mu$.
\end{enumerate}
Similarly, the associativity axiom requires that the following commutes for all~$A$:
\[
\begin{tikzcd}
\Id(\MT^3(A)) \rar{\mu} \dar[swap]{\Id(\mu)} & \MT(\MT(A)) \rar{\MT(\mu)} & \MT^2(A) \dar{\mu} \\
\MT^2(A) \arrow[rr, swap, "\mu"] & & \MT(A)
\end{tikzcd}
\]
That is, the components of~$\mu$ are left ${\mu}$-morphisms of type~$\mu \klar{\mu} \mu$.
\end{example}

\subsection{Kleisli and Eilenberg-Moore Laws}
In later sections, we will require more structure on the morphism~$\lambda$ in diagrams~\eqref{eq:left-bimorphism} and~\eqref{eq:right-bimorphism}. Of particular interest will be those that interact well with monad structure.

Let~$\MS : \CC \rightarrow \CC$ and~$\MT : \CD \rightarrow \CD$ be monads, and~$H : \CC \rightarrow \CD$ a functor. A \df{Kleisli law} is a natural transformation
\[\lambda : H \circ \MS \nt \MT \circ H \]
such that the following diagrams commute:
\begin{equation}
\label{eq:kleisli-law-unit-axiom}
\begin{tikzcd}[column sep=0.5em]
    H\MS(A)
        \ar{rr}{\lambda}
        \ar[<-,swap]{dr}{H\eta}
    & & \MT H(A)
        \ar[<-]{dl}{\eta}
    \\
    &H(A)
\end{tikzcd}
\end{equation}
\begin{equation}
\label{eq:kleisli-law-multiplication-axiom}
\begin{tikzcd}
H\MS^2(A) \dar[swap]{H\mu} \rar{\lambda} & \MT H \MS(A) \rar{\MT \lambda} & \MT^2 H (A) \dar{\mu} \\  
H \MS(A) \arrow[rr, swap, "\lambda"] & & \MT H(A)
\end{tikzcd}
\end{equation}
The significance of Kleisli laws is captured in the following theorem, which is probably folklore. See for example~\cite{manes2007monad}.
\begin{theorem}[Kleisli laws Classify Liftings]
\label{thm:kleisli-correspondence}
Let~$\MS : \CC \rightarrow \CC$ and $\MT : \CD \rightarrow \CD$ be monads, and~$H : \CC \rightarrow \CD$ a functor. There is a bijective correspondence between:
\begin{enumerate}
    \item \label{en:kl-laws} Kleisli laws of type~$H \circ \MS \nt \MT \circ H$.
    \item \label{en:kl-liftings} \emph{Kleisli liftings} $\klc{\MS}{\CC} \rightarrow \klc{\MT}{\CD}$, i.e.\ functors $\overline H : \klc{\MS}{\CC} \rightarrow \klc{\MT}{\CD}$ such that the following diagram commutes:
    \[
        \begin{tikzcd}
            \klc{\MS}{\CC} \rar{\overline{H}}  & \klc{\MT}{\CD} \\
            \CC \uar{\klfree{\MS}} \rar[swap]{H} & \CD \uar[swap]{\klfree{\MT}}
        \end{tikzcd}
    \]
\end{enumerate}
\end{theorem}
For a Kleisli law~$\lambda : H \circ \MS \nt \MT \circ H$, the action on morphisms of its corresponding Kleisli lifting~$\klf{H}{\lambda} : \klc{\MS}{\CC} \rightarrow \klc{\MT}{\CD}$ is:
\[
\klf{H}{\lambda} :\quad A \xrightarrow{f} \MS(B) \quad\mapsto\quad H(A) \xrightarrow{H(f)} H(\MS(B)) \xrightarrow{\lambda} \MT(H(B))
\]

\begin{example}
If~$\MT$ is a strong monad, for each~$A$, the morphisms~$A \otimes \MT(B) \nt \MT(A \otimes B)$ form a Kleisli law.
\end{example}

We will also have need of a dual notion to Kleisli laws. Let~$\MS : \CC \rightarrow \CC$ and~$\MT : \CD \rightarrow \CD$ be monads, and~$G : \CD \rightarrow \CC$ a functor. An \df{Eilenberg-Moore law} is a natural transformation~$\rho : \MS \circ G \nt G \circ \MT$ such that the following diagrams commute:
\[
\begin{tikzcd}[column sep=0.5em]
\MS(G(A))
    \ar{rr}{\rho}
    & & G(\MT(A))
     \\
    & G(A)
    \ar{ul}{\eta}
    \ar[swap]{ur}{G(\eta)}
\end{tikzcd}
\]
\[
\begin{tikzcd}
    \MS^2(G(A)) \dar[swap]{\mu} \rar{\MS(\rho)} & \MS(G(\MT(A))) \rar{\rho} & G(\MT^2(A)) \dar{G\mu} \\
    \MS(G(A)) \arrow[rr, swap, "\rho"] & & G(\MT(A))
\end{tikzcd}
\]
Similarly to Kleisli laws, the significance of Eilenberg-Moore laws is captured by the following well-known result, credited to~\cite{appelgate1965acyclic} in \cite{johnstone1975adjoint}.
\begin{theorem}[Eilenberg-Moore laws Classify Liftings]
\label{thm:eilenberg-moore-correspondence}
Let~$\MS : \CC \rightarrow \CC$ and $\MT : \CD \rightarrow \CD$ be monads, and~$K : \CD \rightarrow \CC$ a functor. There is a bijective correspondence between:
\begin{enumerate}
    \item \label{en:em-laws} Eilenberg-Moore laws of type~$\MS \circ G \nt G \circ \MT$.
    \item \label{en:em-liftings} \emph{Eilenberg-Moore liftings} $\emc{\MT}{\CD} \to \emc{\MS}{\CC}$, i.e. functors $\overline{G} : \emc{\MT}{\CD} \to \emc{\MS}{\CC}$ such that the following diagram commutes:
    \[
        \begin{tikzcd}
            \emc{\MT}{\CD}
                \dar[swap]{\eilmoforget{\MT}}
                \rar{\overline{G}}  &
            \emc{\MS}{\CC}
                \dar{\eilmoforget{\MS}}
            \\
            \CD \rar[swap]{G} & \CC
        \end{tikzcd}
    \]
\end{enumerate}
\end{theorem}
For a given Eilenberg-Moore law~$\rho : \MS \circ G \nt G \circ \MT$, the action on objects of the corresponding Eilenberg-Moore lifting~$\emf{G}{\rho} : \emc{\MT}{\CD} \to \emc{\MS}{\CC}$  is:
\begin{equation}
\label{eq:eilenberg-moore-law-functor-action}
\emf{G}{\rho}: \quad \MT(A) \xrightarrow{\alpha} A \quad\mapsto\quad \MS(G(A)) \xrightarrow{\rho} G(\MT(A)) \xrightarrow{G(\alpha)} G(A) 
\end{equation}
The following well-known facts about Kleisli and Eilenberg-Moore laws will be useful.
\begin{proposition}
\label{prop:em-kleisli-facts}
If~$L : \CC \rightarrow \CD$ is a functor, then:
\begin{enumerate}
    \item \label{en:em-inverse} If~$\rho : \MS \circ L \nt L \circ \MT$ is an invertible Eilenberg-Moore law, then the inverse~$\rho^{-1} : L \circ \MT \nt \MS \circ L$ is a Kleisli law. 
    \item \label{en:em-transpose} If we have adjunction~$L \dashv R$, with unit~$\eta : \Id \nt R \circ L$ and counit~$\epsilon : L \circ R \nt \Id$, then natural transformation~$\rho : \MS \circ R \nt R \circ \MT$ is an Eilenberg-Moore law if and only if its transpose:
    \[ L \circ \MS \xrightarrow{L\MS(\eta)} L \circ \MS \circ R \circ L \xrightarrow{L(\rho_L)} L \circ R \circ \MT \circ L \xrightarrow{\epsilon_{\MT L}} \MT \circ L \]
    is a Kleisli law.
\end{enumerate}
\end{proposition}

\subsection{General Arities}
\label{s:general-arities}
It is natural to consider functors, and corresponding bimorphisms, of more general arities. To do so, we consider categories~$\CC_1,\ldots,\CC_n$ and~$\CD$, with:
\begin{enumerate}
    \item A functor~$H : \CC_1 \times \ldots \times \CC_n \rightarrow \CD$.
    \item Endofunctors 
    \[ S_1 : \CC_1 \rightarrow \CC_1, \, \ldots, \, S_n : \CC_n \rightarrow \CC_n ,\]
    and algebras 
    \[S_1(A_1) \xrightarrow{\alpha_1} A_1, \, \ldots, \, S_n(A_n) \xrightarrow{\alpha_n} A_n . \]
    \item An endofunctor~$T : \CD \rightarrow \CD$ and algebra~$T(B) \xrightarrow{\beta} B$.
\end{enumerate}
We can then generalize condition~\eqref{eq:left-bimorphism} to
\begin{equation}
\label{eq:left-bimorphism-general-arities}
\begin{tikzcd}
H(S_1(A_1),\ldots,S_n(A_n)) \rar{\lambda} \dar[swap]{H(\alpha_1,\ldots,\alpha_n)} & T(H(A_1,\ldots,A_n)) \rar{T(h)} & T(B) \dar{\beta} \\
H(A_1,\ldots,A_n) \arrow[rr, swap, "h"] & & B
\end{tikzcd}
\end{equation}
Assuming further that $\MS_1,\dots,\MS_n$ are monads, we could then define an~$n$-ary Kleisli law as a natural transformation of type
\[
H(\MS_1(A_1),\ldots,\MS(A_n)) \nt \MT(H(A_1,\ldots,A_n))
\]
such that the following two diagrams commute:
\begin{equation}
\label{eq:multi-kleisli-unit}
\begin{tikzcd}[column sep=0.5em]
H(\MS_1(A_1),\ldots, \MS_n(A_n))
    \ar{rr}{\lambda}
& & \MT H(A_1,\ldots,A_n)
\\
&H(A_1,\ldots,A_n)
    \ar{ul}{H(\eta,\ldots,\eta)}
    \ar[swap]{ur}{\eta}
\end{tikzcd}
\end{equation}
\begin{equation}
\label{eq:multi-kleisli-multiplication}
\begin{tikzcd}
H(\MS_1^2(A_1),\ldots,\MS_n^2(A_n)) \dar[swap]{H(\mu,\ldots,\mu)} \rar{\lambda} & \MT H(\MS_1(A_1),\ldots,\MS_n(A_n)) \rar{\MT \lambda} & \MT^2 H(A_1,\ldots,A_n) \dar{\mu} \\  
H(\MS_1(A_1),\ldots,\MS_n(A_n)) \arrow[rr, swap, "\lambda"] & & \MT H(A_1,\ldots,A_n)
\end{tikzcd}
\end{equation}
We note that for a family of endofunctors~$T_i : \CC_i \rightarrow \CC_i$, there is a product endofunctor~$\prod_i T_i : \prod_i \CC_i \rightarrow \prod_i \CC_i$. In the case that the~$T_i$ are monads, $\prod_i T_i$ also carries a monad structure pointwise.
\begin{proposition}
Assume~$\MT_i : \CC_i \rightarrow \CC_i$ is a family of monads. There is an induced monad with:
\begin{description}
\item[Functor]: The endofunctor~$\prod_i \MT_i : \prod_i \CC_i \rightarrow \CC_i$ acts pointwise:
\[
\prod_i \MT_i(A_1,\ldots,A_n) := (\MT_1(A_1),\ldots,\MT_n(A_n))
\]
\item[Unit]: The unit has components:
\[
\eta_{(A_1,\ldots,A_n)} := (\eta_{A_1},\ldots,\eta_{A_n})
\]
\item[Multiplication]: The multiplication has components:
\[
\mu_{(A_1,\ldots,A_n)} := (\mu_{A_1},\ldots,\mu_{A_n})
\]
\end{description}
\end{proposition}
Unravelling definitions, for endofunctors~$T_1, \ldots, T_n$, a~$\prod_i T_i$-algebra is simply a product of~$T_i$-algebras. In the case of monads~$\MT_1, \ldots, \MT_n$
a $\prod_i \MT_i$-algebra $((A_1,\ldots,A_n), (\alpha_1,\ldots,\alpha_n))$ is exactly a tuple with each~$\alpha_i$ a~$\MT_i$-algebra structure map. Furthermore, a natural transformation~$\lambda : H \circ \prod_i \MS_i \nt \MT \circ H$ is a natural family
\[ \lambda : H(\MT_1(A_1),\ldots,\MT_n(A_n)) \rightarrow \MT(H(A_1,\ldots,A_n)) \]
$\lambda$ is a Kleisli law exactly when it satisfies diagrams~\eqref{eq:multi-kleisli-unit} and~\eqref{eq:multi-kleisli-multiplication}.

The fact that an~$n$-ary Kleisli law is the same thing as a Kleisli law for the induced monad on the product category provides two useful perspectives:
\begin{enumerate}
    \item The product monad perspective means we can prove theoretical results in full generality using the simpler formulation, using condition~\eqref{eq:left-bimorphism}.
    \item The explicit condition~\eqref{eq:left-bimorphism-general-arities} given above can clarify concrete situations in applications by exposing the separate components.
\end{enumerate}
A similar strategy can be applied to condition~\eqref{eq:right-bimorphism} for right bimorphisms.

\begin{example}
For any monad, $\dst$ and~$\dstp$ are Kleisli laws of type 
\[ \MT(A) \otimes \MT(B) \rightarrow \MT(A \otimes B) \]
In the terminology of Kleisli laws, for a commutative monad, $h : A \otimes B \rightarrow C$ is a bimorphism exactly if it is a~left ${\dst}$-morphism.
\end{example}

\subsection{Examples of bimorphisms for Kleisli laws}

\begin{example}
For any Kleisli law~$\lambda : H \circ \MS \nt \MT \circ H$, the axiom in diagram~\eqref{eq:kleisli-law-multiplication-axiom}
can be interpreted as saying the components of~$\lambda$ are~$\lambda$-morphisms $\mu \klar{\lambda} \mu$.
\end{example}


\begin{example}
    Let~$\lambda : \MS \circ \MT \nt \MT \circ \MS$ be a distributive law in the sense of Beck~\cite{beck1969distributive}. An algebra for the induced composite monad~$\MT \circ \MS$ consists of
\begin{enumerate}
    \item A $\MS$-algebra $\alpha^S : \MS(A) \rightarrow A$.
    \item A $\MT$-algebra $\alpha^T : \MT(A) \rightarrow A$. 
\end{enumerate}
These must satisfy the following compatibility condition:
\[
\begin{tikzcd}
\MS \MT(A) \dar[swap]{\MS(\alpha^T)} \rar{\lambda} & \MT \MS(A) \rar{ \MT(\alpha^S) } & \MT(A) \dar{\alpha^T} \\
\MS(A) \arrow[rr, swap, "\alpha^S"] & & A
\end{tikzcd}
\]
That is, $\alpha^S : \alpha^T \klar{\lambda} \alpha^T$. 
\end{example}

\begin{example}
\label{ex:monad-morphisms}
Every monad morphism~$\sigma : \MS \nt \MT$ induces:
\begin{enumerate}
    \item A functor~$\klf{\Id}{\sigma} : \klc{\MS}{\CC} \rightarrow \klc{\MT}{\CD}$, identity on objects, with action on morphisms:
    \[ f \mapsto \sigma \circ f \]
    \item A functor~$\emf{\Id}{\sigma} : \emc{\MT}{\CD} \rightarrow \emc{\MS}{\CC}$, identity on morphisms, with action on objects:
    \[
    (A,\alpha) \mapsto (A, \alpha \circ \sigma)
    \]
\end{enumerate}
Such a $\sigma$ is a special case of a Kleisli law, of type~$\Id \circ \MS \nt \MT \circ \Id$. A left $\sigma$-morphism from~$\alpha$ to~$\beta$ satisfies:
\[
\begin{tikzcd}
\MS(A) \rar{\sigma} \dar[swap]{\alpha} & \MT(A) \rar{\MT{h}} & \MT(B) \dar{\beta} \\
A \arrow[rr, swap, "h"] & & B
\end{tikzcd}
\]
By the naturality of~$\sigma$, this is equivalent to:
\[
\begin{tikzcd}
\MS(A) \rar{\MS(h)} \dar[swap]{\alpha} & \MS(B) \rar{\sigma} & \MT(B) \dar{\beta} \\
A \arrow[rr, swap, "h"] & & B
\end{tikzcd}
\]
Which is exactly the condition that $h$ is an~$\MS$-algebra morphism of type~$\alpha \rightarrow \emf{\Id}{\sigma}(\beta)$. 
\end{example}

\section{Bimorphisms Graphically}
\label{s:bim-graph}
We now briefly explore bimorphisms in terms of string diagrams. For background on the notation, see for example~\cite{HinzeM16a, HinzeM16b}.
For $T$-algebras~$(A,\alpha)$ and $(B,\beta)$, $h$ is an algebra morphism of type~$\alpha \rightarrow \beta$ if the following equation holds.
\[
    \begin{tikzpicture}[stringdiagram, scale=0.5, baseline=(anchor)]
    \path coordinate[dot, label=right:$\alpha$] (alpha) +(0,1) coordinate[label=above:$A$] (tr) +(-1,1) coordinate[label=above:$T$] (tl)
    ++(0,-1) coordinate[dot, label=right:$h$] (h)
    ++(0,-1) coordinate[label=below:$B$] (bot);
    \coordinate (anchor) at ($(alpha)!0.5!(h)$);
    \draw 
    (tr) -- (alpha) -- (h) -- (bot)
    (tl) to[out=-90, in=180] (alpha);
    \begin{scope}[on background layer]
    \fill[catterm] (bot) rectangle ($(tr) + (1,0)$);
    \fill[catc] (bot) rectangle ($(tl) + (-1,0)$);
    \end{scope}
    \end{tikzpicture}
    =
    \begin{tikzpicture}[stringdiagram, scale=0.5, baseline=(anchor)]
    \path coordinate[dot, label=right:$h$] (h) 
    +(0,1) coordinate[label=above:$A$] (tr)
    +(-1,1) coordinate[label=above:$T$] (tl)
    +(-1,0) coordinate (a)
    ++(0,-1) coordinate[dot, label=right:$\beta$] (beta)
    ++(0,-1) coordinate[label=below:$B$] (bot);
    \coordinate (anchor) at ($(beta)!0.5!(h)$);
    \draw
    (tr) -- (bot)
    (tl) -- (a.center) to[out=-90, in=180] (beta);
    \begin{scope}[on background layer]
    \fill[catterm] (bot) rectangle ($(tr) + (1,0)$);
    \fill[catc] (bot) rectangle ($(tl) + (-1,0)$);
    \end{scope}
    \end{tikzpicture}
\]
Intuitively, we see this as the ability to slide ``the'' algebra past the homomorphism.

Now let~$(A,\alpha)$ be an~$S$-algebra, and~$(B,\beta)$ a~$T$-algebra, and~$\lambda : H \circ S \nt T \circ H$ a natural transformation. Then, $h$ is a left $\lambda$-morphism if the following equation holds:
\begin{equation}
\label{eq:left-bimorphism-stringdiagrams}
\begin{tikzpicture}[scale=0.5, baseline=(h)]
\path coordinate[dot, label=right:$\alpha$] (alpha)
+(0,1) coordinate[label=above:$A$] (tr)
+(-1,1) coordinate[label=above:$S$] (tm)
++(0,-1) coordinate[dot, label=right:$h$] (h)
+(-2,2) coordinate[label=above:$H$] (tl)
+(0,-2) coordinate[label=below:$B$] (bot);
\draw 
(tr) -- (bot) 
(alpha.center) to[out=-180, in=-90] (tm)
(h.center) to[out=-180, in=-90] (tl);
\begin{scope}[on background layer]
\fill[catd] (bot) rectangle ($(tl) + (-1,0)$);
\fill[catc] (h.center) to[out=-180, in=-90] (tl) -- (tr) -- cycle;
\fill[catterm] (bot) rectangle ($(tr) + (1,0)$);
\end{scope}
\end{tikzpicture}
=
\begin{tikzpicture}[scale=0.5, baseline=(h)]
\path coordinate[dot, label=right:$h$] (h)
+(-1,0) coordinate (a)
+(-2,2) coordinate[label=above:$H$] (tl)
+(-1,2) coordinate[label=above:$S$] (tm)
+(0,2) coordinate[label=above:$A$] (tr)
++(0,-1) coordinate[dot, label=right:$\beta$] (beta)
+(0,-1) coordinate[label=below:$B$] (bot);
\draw (tr) -- (bot);
\draw[name path=curv1] (h.center) to[out=180, in=-90] (tl);
\draw[name path=curv2] (beta.center) to[out=180, in=-90] (a.center) -- (tm);
\path[name intersections={of=curv1 and curv2}]
    coordinate[dot, label=-135:$\lambda$] (lambda) at (intersection-1);
\begin{scope}[on background layer]
\fill[catd] (bot) rectangle ($(tl) + (-1,0)$);
\fill[catc] (h.center) to[out=-180, in=-90] (tl) -- (tr) -- cycle;
\fill[catterm] (bot) rectangle ($(tr) + (1,0)$);
\end{scope}
\end{tikzpicture}
\end{equation}
Our bimorphism~$h : H(A) \rightarrow B$ now has a structured input, and visually the mediating natural transformation
$\lambda$ allows us to ``cross wires'' so that we can slide the algebra over the bimorphism.

If~$(A,\alpha)$ is a~$S$-algebra, and~$(B,\beta)$ an~$T$-algebra, and~$\rho : S \circ G \nt G \circ T$, then~$h : A \rightarrow G(B)$ is a~right $\rho$-morphism if the following equation holds:
\begin{equation}
\label{eq:right-bimorphism-stringdiagrams}
\begin{tikzpicture}[stringdiagram, scale=0.5, baseline=(k)]
\path coordinate[dot, label=right:$h$] (k)
+(-2,-2) coordinate[label=below:$G$] (bl)
+(0,2) coordinate[label=above:$A$] (tr)
+(0,-2) coordinate[label=below:$B$] (br)
++(0,1) coordinate[dot, label=right:$\alpha$] (alpha)
+(-1,1) coordinate[label=above:$S$] (tl);
\draw 
(tr) -- (br)
(alpha.center) to[out=180, in=-90] (tl)
(k.center) to[out=180, in=90] (bl);
\begin{scope}[on background layer]
\fill[catterm] (br) rectangle ($(tr) + (1,0)$);
\fill[catd] (tr) rectangle ($(bl) + (-1,0)$);
\fill[catc] (k.center) to[out=180, in=90] (bl) -- (br) -- cycle;
\end{scope}
\end{tikzpicture}
=
\begin{tikzpicture}[stringdiagram, scale=0.5, baseline=(k)]
\path coordinate[dot, label=right:$h$] (k)
+(-2,-2) coordinate[label=below:$G$] (bl)
+(0,2) coordinate[label=above:$A$] (tr)
+(-1,2) coordinate[label=above:$S$] (tl)
+(0,-2) coordinate[label=below:$B$] (br)
+(-1,0) coordinate (a)
++(0,-1) coordinate[dot, label=right:$\beta$] (beta);
\draw (tr) -- (br);
\draw[name path=curv1] (k.center) to[out=180, in=90] (bl);
\draw[name path=curv2] (beta.center) to[out=180, in=-90] (a) -- (tl);
\path[name intersections={of=curv1 and curv2}]
    coordinate[dot, label=135:$\rho$] (kappa) at (intersection-1);
\begin{scope}[on background layer]
\fill[catterm] (br) rectangle ($(tr) + (1,0)$);
\fill[catd] (tr) rectangle ($(bl) + (-1,0)$);
\fill[catc] (k.center) to[out=180, in=90] (bl) -- (br) -- cycle;
\end{scope}
\end{tikzpicture}
\end{equation}
Now our bimorphism~$h : A \rightarrow G(B)$ has a structured output, and the mediating natural transformation~$\rho$ allows
us to cross wires, so we can slide the algebras over the bimorphism.

The visual nature of the string diagram renditions of the bimorphism conditions~\eqref{eq:left-bimorphism} and~\eqref{eq:right-bimorphism} as~\eqref{eq:left-bimorphism-stringdiagrams} and~\eqref{eq:right-bimorphism-stringdiagrams} gives a more intuitive sense for how the natural transformation is needed for the bimorphism to commute with the
algebra structures. 

Drawing the identity functor explicitly, as a dashed edge,
we see that the ordinary homomorphism condition is a special case of both of the above. Specializing
equation~\eqref{eq:left-bimorphism-stringdiagrams}:
\[
\begin{tikzpicture}[scale=0.5, baseline=(h)]
\path coordinate[dot, label=right:$\alpha$] (alpha)
+(0,1) coordinate[label=above:$A$] (tr)
+(-1,1) coordinate[label=above:$S$] (tm)
++(0,-1) coordinate[dot, label=right:$h$] (h)
+(-2,2) coordinate[label=above:$H$] (tl)
+(0,-2) coordinate[label=below:$B$] (bot);
\draw 
(tr) -- (bot) 
(alpha.center) to[out=-180, in=-90] (tm);
\draw[id functor] (h.center) to[out=-180, in=-90] (tl);
\begin{scope}[on background layer]
\fill[catc] (bot) rectangle ($(tl) + (-1,0)$);
\fill[catterm] (bot) rectangle ($(tr) + (1,0)$);
\end{scope}
\end{tikzpicture}
=
\begin{tikzpicture}[scale=0.5, baseline=(h)]
\path coordinate[dot, label=right:$h$] (h)
+(-1,0) coordinate (a)
+(-2,2) coordinate[label=above:$H$] (tl)
+(-1,2) coordinate[label=above:$S$] (tm)
+(0,2) coordinate[label=above:$A$] (tr)
++(0,-1) coordinate[dot, label=right:$\beta$] (beta)
+(0,-1) coordinate[label=below:$B$] (bot);
\draw (tr) -- (bot);
\draw[name path=curv1, id functor] (h.center) to[out=180, in=-90] (tl);
\draw[name path=curv2] (beta.center) to[out=180, in=-90] (a.center) -- (tm);
\path[name intersections={of=curv1 and curv2}]
    coordinate[id nt] (lambda) at (intersection-1);
\begin{scope}[on background layer]
\fill[catc] (bot) rectangle ($(tl) + (-1,0)$);
\fill[catterm] (bot) rectangle ($(tr) + (1,0)$);
\end{scope}
\end{tikzpicture}
\]
Similarly, specializing condition~\eqref{eq:left-bimorphism-stringdiagrams}:
\[
\begin{tikzpicture}[stringdiagram, scale=0.5, baseline=(k)]
\path coordinate[dot, label=right:$k$] (k)
+(-2,-2) coordinate[label=below:$K$] (bl)
+(0,2) coordinate[label=above:$A$] (tr)
+(0,-2) coordinate[label=below:$B$] (br)
++(0,1) coordinate[dot, label=right:$\alpha$] (alpha)
+(-1,1) coordinate[label=above:$S$] (tl);
\draw 
(tr) -- (br)
(alpha.center) to[out=180, in=-90] (tl);
\draw[id functor] (k.center) to[out=180, in=90] (bl);
\begin{scope}[on background layer]
\fill[catterm] (br) rectangle ($(tr) + (1,0)$);
\fill[catc] (tr) rectangle ($(bl) + (-1,0)$);
\end{scope}
\end{tikzpicture}
=
\begin{tikzpicture}[stringdiagram, scale=0.5, baseline=(k)]
\path coordinate[dot, label=right:$k$] (k)
+(-2,-2) coordinate[label=below:$K$] (bl)
+(0,2) coordinate[label=above:$A$] (tr)
+(-1,2) coordinate[label=above:$S$] (tl)
+(0,-2) coordinate[label=below:$B$] (br)
+(-1,0) coordinate (a)
++(0,-1) coordinate[dot, label=right:$\beta$] (beta);
\draw (tr) -- (br);
\draw[name path=curv1, id functor] (k.center) to[out=180, in=90] (bl);
\draw[name path=curv2] (beta.center) to[out=180, in=-90] (a) -- (tl);
\path[name intersections={of=curv1 and curv2}]
    coordinate[id nt] (kappa) at (intersection-1);
\begin{scope}[on background layer]
\fill[catterm] (br) rectangle ($(tr) + (1,0)$);
\fill[catc] (tr) rectangle ($(bl) + (-1,0)$);
\end{scope}
\end{tikzpicture}
\]
Diagrammatically, it is easy to see that if~$h : \alpha \klar{\lambda} \beta$ and~$g : \beta \klar{\lambda'} \gamma$, then $k \circ G(h) : \alpha \klar{\lambda' \circ G \lambda} \gamma$:
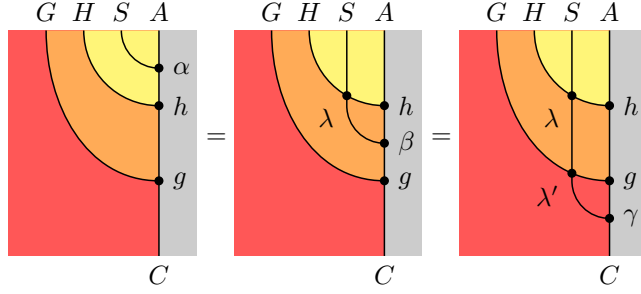
\[
\begin{tikzpicture}[stringdiagram, scale=0.5, baseline=(anchor)]
\path coordinate[dot, label=right:$\alpha$] (alpha)
+(0,1) coordinate[label=above:$A$] (tr)
+(-1,1) coordinate[label=above:$S$] (tm)
++(0,-1) coordinate[dot, label=right:$h$] (h)
+(-2,2) coordinate[label=above:$H$] (tl)
++(0,-2) coordinate[dot, label=right:$g$] (k)
+(-3,4) coordinate[label=above:$G$] (tll)
+(0,-2) coordinate[label=below:$C$] (bot);
\coordinate (anchor) at ($(h)!0.5!(k)$);
\draw 
(bot) -- (tr)
(alpha.center) to[out=180, in=-90] (tm)
(h.center) to[out=180, in=-90] (tl)
(k.center) to[out=180, in=-90] (tll);
\begin{scope}[on background layer]
\fill[catterm] (bot) rectangle ($(tr) + (1,0)$);
\fill[cate] (bot) rectangle ($(tll) + (-1,0)$);
\fill[catd] (k.center) to[out=180, in=-90] (tll) -- (tr) -- cycle;
\fill[catc] (h.center) to[out=180, in=-90] (tl) -- (tr) -- cycle;
\end{scope}
\end{tikzpicture}
=
\begin{tikzpicture}[stringdiagram, scale=0.5, baseline=(anchor)]
\path coordinate[dot, label=right:$h$] (h)
+(-1,0) coordinate (a)
+(0,2) coordinate[label=above:$A$] (tr)
+(-1,2) coordinate[label=above:$S$] (tm)
+(-2,2) coordinate[label=above:$H$] (tl)
++(0,-1) coordinate[dot, label=right:$\beta$] (beta)
++(0,-1) coordinate[dot, label=right:$g$] (k)
+(-3,4) coordinate[label=above:$G$] (tll)
+(0,-2) coordinate[label=below:$C$] (bot);
\coordinate (anchor) at ($(h)!0.5!(k)$);
\draw (bot) -- (tr);
\draw[name path=curv1] (h.center) to[out=180, in=-90] (tl);
\draw[name path=curv2] (beta.center) to[out=180, in=-90] (a.center) -- (tm);
\path[name intersections={of=curv1 and curv2}]
    coordinate[dot, label=-135:$\lambda$] (lambda) at (intersection-1);
\draw (k.center) to[out=180, in=-90] (tll);
\begin{scope}[on background layer]
\fill[catterm] (bot) rectangle ($(tr) + (1,0)$);
\fill[cate] (bot) rectangle ($(tll) + (-1,0)$);
\fill[catd] (k.center) to[out=180, in=-90] (tll) -- (tr) -- cycle;
\fill[catc] (h.center) to[out=180, in=-90] (tl) -- (tr) -- cycle;
\end{scope}
\end{tikzpicture}
=
\begin{tikzpicture}[stringdiagram, scale=0.5, baseline=(anchor)]
\path coordinate[dot, label=right:$h$] (h)
+(0,2) coordinate[label=above:$A$] (tr)
+(-1,2) coordinate[label=above:$S$] (tm)
+(-2,2) coordinate[label=above:$H$] (tl)
++(0,-2) coordinate[dot, label=right:$g$] (k)
+(-3,4) coordinate[label=above:$G$] (tll)
+(0,-2) coordinate[label=below:$C$] (bot)
+(-1,0) coordinate (a)
++(0,-1) coordinate[dot, label=right:$\gamma$] (gamma);
\coordinate (anchor) at ($(h)!0.5!(k)$);
\draw (bot) -- (tr);
\draw[name path=curv1] (h.center) to[out=180, in=-90] (tl);
\draw[name path=curv2] (k.center) to[out=180, in=-90] (tll);
\draw[name path=curv3] (gamma.center) to[out=180, in=-90] (a.center) -- (tm);
\path[name intersections={of=curv1 and curv3}]
    coordinate[dot, label=-135:$\lambda$] (lambda) at (intersection-1);
\path[name intersections={of=curv2 and curv3}]
    coordinate[dot, label=-135:$\lambda'$] (kappa) at (intersection-1);    
\begin{scope}[on background layer]
\fill[catterm] (bot) rectangle ($(tr) + (1,0)$);
\fill[cate] (bot) rectangle ($(tll) + (-1,0)$);
\fill[catd] (k.center) to[out=180, in=-90] (tll) -- (tr) -- cycle;
\fill[catc] (h.center) to[out=180, in=-90] (tl) -- (tr) -- cycle;
\end{scope}
\end{tikzpicture}
\]
As a special case of the above observation, we note that if~$m$ is a~left $\lambda$-morphism of type $\alpha \klar\lambda \beta$, and~$h : \alpha' \rightarrow \alpha$ and~$k : \beta \rightarrow \beta'$ are algebra morphisms, then~$k \circ m \circ H(h) : \alpha' \klar{\lambda} \beta'$. 


\section{Universal Constructions}
The results of this section generalize the lifting of monoidal structure in the case of
commutative monads discussed in section~\ref{sec:lifting-tensor-products}.
Probably the most explicit proofs for the analogous commutative monads constructions appear in~\cite{seal2012tensors}.
We wish to acknowledge this paper had great impact on our own approach.

\subsection{Classifying Objects for \texorpdfstring{$H_\lambda$-morphisms}{H lambda-morphisms}}

Let~$\MS : \CC \rightarrow \CC$ and $\MT : \CD \rightarrow \CD$ be monads, $H : \CC \to \CD$ a functor and \mbox{$\lambda : H(\MS(A)) \rightarrow \MT(H(A))$} a morphism in~$\CD$. An important feature of bimorphisms is that under sufficient cocompleteness conditions, they are in bijection with ordinary algebra homomorphisms of a suitable type.
Specifically, given an \mbox{$\MS$-algebra} $(A,\alpha)$, if the following coequalizer exists in~$\emc{\MT}{\CD}$:
\begin{equation}
    \label{eq:lifting-coeq-pair}
    \begin{tikzcd}
        \emfr{\MT}(H(\MS(A)))
        \ar[yshift=+0.5em]{rr}{\mu^\MT \circ \emfr{\MT}(\lambda)}
        \ar[swap,yshift=-0.5em]{rr}{\emfr{\MT}(H(\alpha))}
        &
        &
        \emfr{\MT}(H(A))
        \rar{q_\alpha}
        &
        (W_\alpha,\omega_\alpha)
    \end{tikzcd}
\end{equation}
then bimorphisms~$\alpha \klar{\lambda} \beta$
are in bijection with $\MT$-algebra morphisms $\omega_\alpha \to \beta$, with $(W_\alpha,\omega_\alpha)$ as defined in diagram~\eqref{eq:lifting-coeq-pair}.

This statement is made precise in the following theorem.
\begin{theorem}[Universal Bimorphisms]
\label{thm:universal-bimorphisms}
    Assuming the coequalizer in \eqref{eq:lifting-coeq-pair} exists, the composite morphism
    \[
        u := H(A) \xrightarrow{\eta^\MT} \emfr{\MT}(H(A)) \xrightarrow{q_\alpha} W_\alpha
    \]
    is a left $\lambda$-morphism $u : \alpha \klar\lambda \omega_\alpha$. Furthermore, it is universal in the sense that every every~$\MT$-algebra~$(B,\beta)$ and $h : \alpha \klar{\lambda} \beta$ there exists a unique~$\MT$-algebra morphism ${\widehat{h} : \omega_\alpha \to \beta}$ such that the following commutes:
\begin{equation*}
    \label{eq:universal-bimorphism-UP}
    \begin{tikzcd}
    H(A) \drar[swap]{h} \rar{u} & W_\alpha \dar{\widehat{h}} \\
    & B
    \end{tikzcd}
\end{equation*}
\end{theorem}
\begin{proof}
Firstly, we must show~$u : H(A) \rightarrow W_\alpha$ is a bimorphism. We calculate:
\begin{eqproof}
\omega_\alpha \circ \MT(q_\alpha) \circ \MT(\eta) \circ \lambda
\explain{ $q_\alpha$ is a $\MT$-algebra homomorphism }
q_\alpha \circ \mu \circ T(\eta) \circ \lambda
\explain{ monad unit axioms }
q_\alpha \circ \mu \circ \eta \circ \lambda
\explain{ $\eta$-naturality }
q_\alpha \circ \mu \circ \MT(\lambda) \circ \eta
\explain{ coequalizer }
    q_\alpha \circ \MT(H(\alpha)) \circ \eta
\explain{ $\eta$-naturality }
q_\alpha \circ \eta \circ H(\alpha)
\end{eqproof}
We then note that if~$h : \alpha \klar{\lambda} \beta$ then $\beta \circ \eilmofree{\MT}(h)$ equalizes the parallel pair in diagram~\eqref{eq:lifting-coeq-pair}. Reasoning in the base category:
\begin{eqproof}
\beta \circ \MT(h) \circ \mu \circ \MT(\lambda)
\explain{ $\mu$-naturality }
\beta \circ \mu \circ \MT^2(h) \circ \MT(\lambda)
\explain{ Eilenberg-Moore algebra multiplication axiom }
\beta \circ \MT(\beta) \circ \MT^2(h) \circ \MT(\lambda)
\explain{ bimorphism assumption plus functoriality of~$\MT$ }
\beta \circ \MT(h) \circ \MT(H(\alpha))
\end{eqproof}
    Therefore by the coequalizer universal property, there is an induced algebra morphism~$\widehat{h} : \omega_\alpha \rightarrow \beta$ such that $\widehat h \circ q_\alpha = \beta \circ \emfr{\MT}(h)$. Conversely, given any algebra morphism~$k : \omega_\alpha \rightarrow \beta$, $k \circ u$ is a bimorphism, as bimorphisms compose and algebra morphisms are bimorphisms too (cf.\ the end of section~\ref{s:bim-graph}).

We aim to show the mappings~$h \mapsto \widehat{h}$ and~$k \mapsto k \circ u$ are mutually inverse. In one direction:
\begin{eqproof}
\widehat{h} \circ u
\explain{ definition }
\widehat{h} \circ q_\alpha \circ \eta
\explain{ coequalizer }
\beta \circ \MT(h) \circ \eta
\explain{ $\eta$-naturality }
\beta \circ \eta \circ h
\explain{ Eilenberg-Moore algebra unit axiom }
h
\end{eqproof}
In the other direction, as:
\begin{eqproof}
\widehat{k \circ u} \circ q_\alpha
\explain{ definition of $u$ and the coequalizer }
\beta \circ \MT(k) \circ \MT(q_\alpha) \circ \MT(\eta)
\explain{ algebra homomorphism }
k \circ \omega_\alpha \circ \MT(q_\alpha) \circ \MT(\eta)
\explain{ algebra homomorphism }
k \circ q_\alpha \circ \mu \circ \MT(\eta)
\explain{ monad unit axiom }
k \circ q_\alpha
\end{eqproof}
the coequalizer universal property completes the proof.
\end{proof}

Note that, in view of the discussion at the end of section~\ref{s:bim-graph}, the diagram in theorem~\ref{thm:universal-bimorphisms} expresses that the bimorphisms $h : \alpha \klar\lambda \beta$ decompose as the composition of bimorphisms $u : \alpha \klar\lambda \omega_\alpha$ and $\widehat h : \omega_\alpha \klar\id \beta$.


\begin{theorem}[Functorial Universal Bimorphisms]
\label{thm:functorial-universal-bimorphisms}
If $\lambda$ is in fact a natural transformation $H \circ \MS \nt \MT \circ H$, and the coequalizer~\eqref{eq:lifting-coeq-pair} exists at~$\alpha$, for every $\MS$-algebra $(A,\alpha)$, then~$\emklf{H}{\lambda}$ extends to a functor~$\emc{\MS}{\CC} \rightarrow \emc{\MT}{\CD}$.
\end{theorem}


This theorem follows from the following localised version of this statement.

\begin{proposition}
If~$(A,\alpha)$ and~$(A',\alpha')$ are~$\MS$-algebras, the coequalizer~\eqref{eq:lifting-coeq-pair} exists at~$\alpha$ and $\alpha'$, and
there exist~$\lambda_A : H(\MS(A)) \rightarrow \MT(H(A))$ and $\lambda_{A'} : H(\MS(A')) \rightarrow \MT(H(A'))$,
every~$\emc{\MS}{\CC}$-morphism $f : \alpha \rightarrow \alpha'$ such that~$\lambda$ is natural with respect to~$f$, in that the following commutes:
\[
\begin{tikzcd}
H(\MS(A)) \rar{\lambda_A} \dar[swap]{H(\MS(f))} & \MT(H(A)) \dar{\MT(H(f)) } \\
H(\MS(A')) \rar[swap]{\lambda_{A'}} & \MT(H(A'))
\end{tikzcd}
\]
    induces a~$\emc{\MT}{\CD}$-morphism $\emklf{H}{\lambda}(f) : \emklf{H}{\lambda}(\alpha) \rightarrow \emklf{H}{\lambda}(\alpha)$. Moreover, $\emklf{H}{\lambda}(\id) = \id$ and $\emklf{H}{\lambda}(f \circ g) = \emklf{H}{\lambda}(f) \circ \emklf{H}{\lambda}(g)$ whenever $\emklf{H}{\lambda}(f)$ and $\emklf{H}{\lambda}(g)$ are induced this way.
\end{proposition}
\begin{proof}
In diagram~\eqref{eq:functoriality-construction} below, the two rows are coequalizer diagrams.
\begin{equation}
\label{eq:functoriality-construction}
    \begin{tikzcd}[row sep=1cm]
    \eilmofree{\MT}(H(\MT(A))) 
    \rar[yshift=0.5em]{\mu \circ \eilmofree{\MT}(\lambda_A)} 
    \rar[yshift=-0.5em, swap]{\eilmofree{\MT}(H(\alpha))}
    & \eilmofree{\MT}(H(A)) \dar{\eilmofree{\MT}(H(f))} \rar{q_\alpha} & (W_\alpha,\omega_\alpha) \\
    \eilmofree{\MT}(H(\MT(A'))) 
    \rar[yshift=0.5em]{\mu \circ \eilmofree{\MT}(\lambda_{A'})} 
    \rar[yshift=-0.5em, swap]{\eilmofree{\MT}(H(\alpha'))}
    & \eilmofree{\MT}(H(A')) \rar[swap]{q_{\alpha'}} & (W_{\alpha'},\omega_{\alpha'})
    \end{tikzcd}
\end{equation}
We wish to show that~$q_{\alpha'} \circ \eilmofree{\MT}(H(f))$ equalizes the parallel pair in the top of the diagram.
\begin{eqproof}
q_{\alpha'} \circ \MT(H(f)) \circ \mu \circ \MT(\lambda_A)
\explain{ $\mu$-naturality }
q_{\alpha'} \circ \mu \circ \MT^2(H(f)) \circ \MT(\lambda_A)
\explain{ $\MT$-functoriality }
q_{\alpha'} \circ \mu \circ \MT(\MT(H(f)) \circ \lambda_A))
\explain{ $\lambda$-naturality with respect to~$f$ }
q_{\alpha'} \circ \mu \circ \MT(\lambda_{A'} \circ H(\MS(f)))
\explain{ $\MT$-functoriality }
q_{\alpha'} \circ \mu \circ \MT(\lambda_{A'}) \circ \MT(H(\MS(f))))
\explain{ coequalizer }
q_{\alpha'} \circ \MT(H(\alpha')) \circ \MT(H(\MS(f))))
\explain{ functoriality twice }
q_{\alpha'} \circ \MT(H(\alpha' \circ \MS(f)))
\explain{ algebra homomorphism }
q_{\alpha'} \circ \MT(H(f \circ \alpha))
\explain{ functoriality twice }
q_{\alpha'} \circ \MT(H(f)) \circ \MT(H(\alpha))
\end{eqproof}
This therefore induces a morphism~$\emklf{H}{\lambda}(f) : (W_\alpha,\omega_\alpha) \rightarrow (W_{\alpha'}, \omega_{\alpha'})$. By the universal property of coequalizers, we clearly have~$\emklf{H}{\lambda}(\id) = \id$.
If we consider two algebra homomorphism~$f : \alpha \rightarrow \alpha'$ and~$g : \alpha' \rightarrow \alpha''$,
and morphisms~$\lambda_A, \lambda_{A'}$ and~$\lambda_{A'}$ which are natural with respect to~$f$ and~$g$, then consider
the diagram below:
\begin{equation*}
\label{eq:functoriality-compositionality}
    \begin{tikzcd}[row sep=1cm]
    \eilmofree{\MT}(H(\MT(A))) 
    \rar[yshift=0.5em]{\mu \circ \eilmofree{\MT}(\lambda_A)} 
    \rar[yshift=-0.5em, swap]{\eilmofree{\MT}(H(\alpha))}
    & \eilmofree{\MT}(H(A)) \dar{\eilmofree{\MT}(H(f))} \rar{q_\alpha} & (W_\alpha,\omega_\alpha) \\
    \eilmofree{\MT}(H(\MT(A'))) 
    \rar[yshift=0.5em]{\mu \circ \eilmofree{\MT}(\lambda_{A'})} 
    \rar[yshift=-0.5em, swap]{\eilmofree{\MT}(H(\alpha'))}
    & \eilmofree{\MT}(H(A')) \dar{\eilmofree{\MT}(H(g))} \rar{q_{\alpha'}} & (W_{\alpha'},\omega_{\alpha'}) \\
    \eilmofree{\MT}(H(\MT(A''))) 
    \rar[yshift=0.5em]{\mu \circ \eilmofree{\MT}(\lambda_{A''})} 
    \rar[yshift=-0.5em, swap]{\eilmofree{\MT}(H(\alpha''))}
    & \eilmofree{\MT}(H(A')) \rar[swap]{q_{\alpha'}} & (W_{\alpha'},\omega_{\alpha'})
    \end{tikzcd}
\end{equation*}
Then we must have~$\emklf{H}{\lambda}(g \circ f) = \emklf{H}{\lambda}(g) \circ \emklf{H}{\lambda}(f)$, as both are suitable fill-in morphisms satisfying the universal property of coequalizers from the top row to the bottom.
\end{proof}
\begin{theorem}[Universal Bimorphisms and Free Constructions]
\label{thm:free-construction}
If~$\lambda : H \circ \MS \nt \MT \circ H$ is a Kleisli law and satisfies the assumptions of theorem~\ref{thm:functorial-universal-bimorphisms}, then there is a natural isomorphism:
\begin{equation*}
    \label{eq:commuting-with-free-algebras}
    \eilmofree{\MT}(H(A)) \cong \emklf{H}{\lambda}(\eilmofree{\MS}(A))
\end{equation*}
\end{theorem}
\begin{proof}
Beyond what is proved in~\ref{prop:free-construction-local}, we must establish that the isomorphism is natural. So we require the following diagram to commute:
\begin{equation}
\label{eq:nat}
\begin{tikzcd}
    \widehat{H}(\eilmofree{\MS}(A)) \rar{\iota_A} \dar[swap]{\widehat{H}(\eilmofree{\MS}(h))} & \eilmofree{\MT}(H(A)) \dar{\eilmofree{\MS}(H(h))} \\
    \widehat{H}(\eilmofree{\MS}(B)) \rar[swap]{\iota_B} & \eilmofree{\MT}(H(B))
\end{tikzcd}
\end{equation}
Here the morphisms $\iota_A$ are the induced isomorphisms from proposition~\ref{prop:free-construction-local}. The following diagram commutes:
\[
\begin{tikzcd}
    \widehat{H}(\eilmofree{\MS}(A)) \dar[swap]{\widehat{H}(\eilmofree{\MS}(h))} & & 
    \arrow[ll, swap, "q_{\mu_A}"] \MT(H(\MS(A))) 
    \arrow[rr, "q_{\mu_A}"] 
    \dlar[swap]{\MT(H(\MS(h)))}
    \drar{\MT(\lambda_A)}
    & & \widehat{H}(\eilmofree{\MS}(A)) \dar{\iota_A} \\
    \widehat{H}(\eilmofree{\MS}(B)) \dar[swap]{\iota_B} & \lar{q_{\mu_B}} \drar[swap]{\MT(\lambda_B)} \MT(H(\MS(A))) & & \MT^2(H(A)) \rar[swap]{\mu_{H(A)}} \dlar{\MT^2(H(h))} & \eilmofree{\MT}(H(A))  \dar{\eilmofree{\MT}(H(h))} \\
    \eilmofree{\MT}(H(B)) & & \arrow[ll, "\mu_{H(B)}"] \MT(H(B)) \arrow[rr, swap, "\mu_{H(B)}"] & & \eilmofree{\MT}(H(B))
\end{tikzcd}
\]
The diamond and the bottom right trapezium commute by naturality, and the remaining parts by the definitions of the coequalizer constructions involved. As the paths from the centre top of the diagram to the bottom corners are equal, both the left and right hand side are candidate coequalizer fill-in morphisms, induced by the same morphism. By the universal property of coequalizers, both the  sides of the diagram, and hence both paths around diagram~\eqref{eq:nat}, are equal.
\end{proof}
Again, we recover the construction from a local version.
\begin{proposition}
\label{prop:free-construction-local}
If
\[ \lambda_A : H(\MS(A) \rightarrow \MT(H(A)) \quad\mbox{ and }\quad \lambda_{\MS(A)} : H(\MS^2(A)) \rightarrow \MT(H(\MS(A)) \] 
are $\CD$-morphisms satisfying Kleisli axioms~\eqref{eq:kleisli-law-unit-axiom}
    and~\eqref{eq:kleisli-law-multiplication-axiom} at~$A$\footnote{That is, $\eta^\MT_{H(A)} = \lambda_A \circ H(\eta^\MS_A)$ and $\mu^\MT_{H(A)} \circ \MT(\lambda_A) \circ \lambda_{\MS(A)} = \lambda_A \circ H(\mu^\MS_A)$.} and natural with respect to $\eta^\MS_A\colon A \to \MS(A)$,
then the coequalizer defining~$\emklf{H}{\lambda}(\eilmofree{\MT}(A))$ exists, and is isomorphic to~$\eilmofree{\MT}(H(A))$.
\end{proposition}
\begin{proof}
We aim to show that~\eqref{eq:free-coequalizer} is a coequalizer diagram:
\begin{equation}
\label{eq:free-coequalizer}
    \begin{tikzcd}
    \eilmofree{\MT}(H(\MS^2(A))) 
    \rar[yshift=0.5em]{\mu \circ \eilmofree{\MT}(\lambda_{\MS(A)})}
    \rar[yshift=-0.5em, swap]{ \eilmofree{\MT}(H(\mu))}
    & \eilmofree{\MT}(H(\MS(A))) \rar{\eilmofree{\MT}(\lambda_A)}
    & \eilmofree{\MT}(\MT(H(A))) \rar{\mu}
    & \eilmofree{\MT}(H(A))
    \end{tikzcd}
\end{equation}
Firstly, to see~$\mu \circ \eilmofree{\MT}(\lambda_A)$ coequalizes the parallel pair:
\begin{eqproof}
\mu \circ \MT(\lambda_A) \circ \mu \circ \MT(\lambda_{\MS(A)})
\explain{ $\mu$-naturality }
\mu \circ \mu \circ \MT^2(\lambda_A) \circ \MT(\lambda_{\MS(A)})
\explain{ monad associativity axiom }
\mu \circ \MT(\mu) \circ \MT^2(\lambda_A) \circ \MT(\lambda_{\MS(A)})
\explain{ $\MT$-functoriality }
\mu \circ \MT(\mu \circ \MT(\lambda_A) \circ \lambda_{\MS(A)})
\explain{ Kleisli axiom \eqref{eq:kleisli-law-multiplication-axiom} at~$A$}
\mu \circ \MT(\lambda_A \circ H(\mu))
\explain{ $\MT$-functoriality }
\mu \circ \MT(\lambda_A) \circ \MT(H(\mu))
\end{eqproof}
If algebra morphism~$f : \eilmofree{\MT}(H(\MS(A)) \rightarrow (B,\beta)$ coequalizes the
parallel pair in diagram~\eqref{eq:free-coequalizer} then~$f = f \circ \eilmofree{\MT}(H(\eta)) \circ \mu \circ \eilmofree{\MT}(\lambda_A)$ as:
\begin{eqproof}
f \circ \MT(H(\eta)) \circ \mu \circ \MT(\lambda_A)
\explain{ $\mu$-naturality }
f \circ \mu \circ \MT^2(H(\eta)) \circ \MT(\lambda_A)
\explain{ $\MT$-functoriality }
f \circ \mu \circ \MT(\MT(H(\eta)) \circ \lambda_A)
\explain{ naturality of $\lambda$ wrt $\eta$ }
f \circ \mu \circ \MT(\lambda_{\MS(A)} \circ H(\MS(\eta)))
\explain{ $\MT$-functoriality }
f \circ \mu \circ \MT(\lambda_{\MS(A)}) \circ \MT(H(\MS(\eta)))
\explain{ $f$ coequalizes $\mu \circ \MT(\lambda_{\MS(A)})$ and $\MT(H(\mu))$}
f \circ \MT(H(\mu)) \circ \MT(H(\MS(\eta)))
\explain{ functoriality and  monad unit axiom }
f
\end{eqproof}
We then note that~$\mu \circ \MT(\lambda_A)$ is a split epimorphism, as:
\begin{eqproof}
\mu \circ \MT(\lambda_A) \circ \MT(H(\eta))
\explain{ $\MT$-functoriality }
\mu \circ \MT(\lambda_A \circ H(\eta))
\explain{ Kleisli axiom \eqref{eq:kleisli-law-unit-axiom} at~$A$ }
\mu \circ \MT(\eta)
\explain{ monad unit axiom }
\id
\end{eqproof}
Therefore, diagram~\eqref{eq:free-coequalizer} satisfies the universal property of a coequalizer diagram.
\end{proof}


By theorem~\ref{thm:free-construction}, we may view $\emklf{H}{\lambda}$ as the lifting of $\klf{H}{\lambda}$ to the category of algebras, as depicted in the following commutative diagram (up to isomorphism). This is gives us a generalised version of theorem~\ref{thm:lifting-smc-structure}.
\[
    \begin{tikzcd}
        \emc{\MS}{\CC} \rar{\emklf{H}{\lambda}} & \emc{\MT}{\CD} \\
        \klc{\MS}{\CC} \uar[pos=0.45]{K} \rar{\klf{H}{\lambda}} & \klc{\MT}{\CD} \uar[swap,pos=0.45]{K} \\
        \CC \ar[bend left=60]{uu}{\emfr{\MS}}\rar{H}\uar{\klfr{\MS}} & \CD \uar[swap]{\klfr{\MT}} \ar[swap,bend right=60]{uu}{\emfr{\MT}}
    \end{tikzcd}
\]

Observe that under the assumptions that $\lambda$ is a Kleisli law, the coequalizer in~\eqref{eq:lifting-coeq-pair} is in fact a coequalizers of reflexive pairs. Indeed, $\MT(H(\eta))$ is always a section of $\MT(H(\alpha))$ and by~\eqref{eq:kleisli-law-unit-axiom} the same holds for $\mu^\MT \circ \emfr{\MT}(\lambda)$. Although the existences of coequalizers in Eilenberg-Moore categories is by no means automatic, standard conditions under which they do exist are known. See for example the accounts in \cite{barr2000toposes}, \cite[Chapter 4]{borceux1994handbook}, and~\cite[Chapter 5]{pedicchio2004categorical}.



\begin{example}[Lifting Binary Coproducts]
Let $\MT$ a monad on a category~$\CC$ with binary coproduct. For every pair $A$, $B$ of $\CC$ objects, there is a canonical morphism:
\[
[\MT(\kappa_1), \MT(\kappa_2)] : T(A) + T(B) \rightarrow T(A + B) ,
\]
where the~$\kappa_i$ are the coproduct injections, and~$[f,g]$ is the morphism induced by~$f$ and~$g$ by the coproduct universal property. This is clearly natural in~$A$ and~$B$. In fact it is a Kleisli morphism. For the unit axiom we calculate:
\begin{eqproof}
[\MT(\kappa_1), \MT(\kappa_2)] \circ \eta + \eta
\explain{ coproducts }
[\MT(\kappa_1) \circ \eta, \MT(\kappa_2) \circ \eta]
\explain{ naturality }
[\eta \circ \kappa_1, \eta \circ \kappa_2]
\explain{ coproducts }
\eta \circ [\kappa_1, \kappa_2]
\explain{ coproducts }
\eta
\end{eqproof}
For the multiplication axiom:
\begin{eqproof}
[\MT(\kappa_1), \MT(\kappa_2)] \circ \mu + \mu
\explain{ coproducts }
[\MT(\kappa_1) \circ \mu, \MT(\kappa_2) \circ \mu]
\explain{ naturality }
[\mu \circ \MT^2(\kappa_1), \mu \circ \MT^2(\kappa_2)]
\explain{ coproducts }
\mu \circ [\MT^2(\kappa_1), \MT^2(\kappa_2)]
\explain{ coproducts }
\mu \circ [\MT[\MT(\kappa_1), \MT(\kappa_2)] \circ \MT(\kappa_1), \MT[\MT(\kappa_1), \MT(\kappa_2)] \circ \MT(\kappa_2)]
\explain{ coproducts }
\mu \circ \MT[\MT(\kappa_1), \MT(\kappa_2)] \circ [\MT(\kappa_1), \MT(\kappa_2)]
\end{eqproof}
The bimorphism condition with respect to the canonical morphism is:
\[
\begin{tikzcd}[column sep=1.75cm]
\MT(A) + \MT(B) \dar[swap]{\alpha + \beta} \rar{[\MT(\kappa_1) + \MT(\kappa_2)]} & \MT(A + B) \rar{\MT(h)} & \MT(C) \dar{\gamma} \\
A + B \arrow[rr, swap, "h"] & & C
\end{tikzcd}
\]
We can define~$h_1 = h \circ \kappa_1$ and~$h_2 = h \circ \kappa_2$, and by the universal property of coproducts, $h = [h_1,h_2]$, so we can rewrite our diagram as:
\[
\begin{tikzcd}[column sep=1.75cm]
\MT(A) + \MT(B) \dar[swap]{\alpha + \beta} \rar{[\MT(\kappa_1) + \MT(\kappa_2)]} & \MT(A + B) \rar{\MT([h_1,h_2])} & \MT(C) \dar{\gamma} \\
A + B \arrow[rr, swap, "{[h_1, h_2]}"] & & C
\end{tikzcd}
\]
which commutes if and only if~$h_1 : \alpha \rightarrow \gamma$ and~$h_2 : \beta \rightarrow \gamma$. If~$\emc{\MT}{\CC}$ has coequalizers of reflexive pairs, we have bijective correspondences:
\begin{prooftree}
\AxiomC{$h_1 : \alpha \rightarrow \gamma$}
\AxiomC{$h_2 : \beta \rightarrow \gamma$}
\doubleLine
\BinaryInfC{$h : (\alpha,\beta) \klar{[\MT \kappa_1, \MT \kappa_2]} \gamma$}
\doubleLine
\UnaryInfC{$\alpha \widehat{+} \beta \rightarrow \gamma$}
\end{prooftree}
This establishes that~$\widehat{+}$ is the coproduct in~$\emc{\MT}{\CC}$. 
\end{example}


\subsection{Classifying Objects for \texorpdfstring{$H^\rho$-morphisms}{H rho-morphisms}}
Given the straightforward functor action induced by Eilenberg-Moore laws~\eqref{eq:eilenberg-moore-law-functor-action}, relating right bimorphisms with respect to an Eilenberg-Moore law
\[ \rho : \MS \circ G \nt G \circ \MT \]
to ordinary algebra morphisms is relatively trivial. We observe that for algebras~$(A,\alpha)$ and~$(B,\beta)$, the following are equivalent for a morphism~$h : A \rightarrow G(B)$:
\begin{enumerate}
    \item $h$ is an algebra morphism $\alpha \rightarrow \emf{G}{\rho}(\beta)$.
    \item $h$ a right ${\rho}$-morphism $\alpha \emar{\rho} \beta$.
\end{enumerate}
Observe that the identity $\id : G(B) \to G(B)$ is a bimorphism $\emf{G}{\rho}(\beta) \to^\rho \beta$.
Consequently, any bimorphism $h : \alpha \to^\rho \beta$ is equal to the composition of bimorphisms:
\[
    \alpha \xrightarrow{\enspace h\enspace}{\!}^\id \enspace\emf{G}\rho(\beta)\enspace \xrightarrow{\enspace\id\enspace}{\!}^\rho \enspace\beta
\]
Note that this is true even in the more general setting where $T$ and $S$ are only required to be endofunctors and $\rho$ is just a natural transformation inducing a functor $\emf{G}{\rho} : \alg T \to \alg S$ between categories of endofunctor algebras.





\section{Dualizing to Comonads}
\label{sec:comonads}
Naturally, as with every categorical concept, the notions introduced in the previous sections dualize. Although dualizing is routine, we provide explicit description of the main results for concreteness.
The results for Kleisli laws of comonads have applications in the emerging theory of \emph{game comonads}~\cite{abramsky2021relating}, which motivated our original investigations.

First, we dualize the notion bimorphism suitable for endofunctor coalgebras. Because of dualizing, the Kleisli laws are stated in terms of right instead of left bimorphisms. Namely, given functors $G : \CC \to \CD$, $C : \CC \to \CC$ and $D : \CD \to \CD$ and a morphism $\rho : D(G(B)) \to G(C(B))$ in $\CD$, we say that $h\colon A\to G(B)$ is a \df{(coalgebraic right) $\rho$-morphisms}, or just \df{bimorphism}, from a $D$-coalgebra $(A,\alpha)$ to a $C$-coalgebra $(B,\beta)$, if it makes the following diagram commute.
\begin{equation}
\begin{tikzcd}
    A \ar[rr, "h"] \dar[swap]{\alpha} & & G(B)  \dar{G\beta}
    \\
    D(A)  \rar{D(h)} & D(G(B)) \rar{\rho} & G(C(B))
\end{tikzcd}
\end{equation}

Further, if $\MC, \MD$ are comonads then a natural transformation $\rho : \MD \circ G \nt G \circ \MC$ is a \df{Kleisli law} if the following commute:
\begin{equation*}
\label{eq:co-kleisli-law-unit-axiom}
\begin{tikzcd}[column sep=0.5em]
    \MD(G(A))
        \ar{rr}{\rho}
        \ar[swap]{dr}{\epsilon}
    &
    & G(\MC(A))
        \ar{dl}{K\epsilon}
    \\
    & G(A)
\end{tikzcd}
\end{equation*}
\begin{equation*}
\label{eq:co-kleisli-law-multiplication-axiom}
\begin{tikzcd}
    \MD(G(A)) \ar[rr, "\rho"] \dar[swap]{\delta} & & G(\MC(A))  \dar{G\delta}
    \\
    \MD^2(G(A)) \rar{\MD(\rho)} & \MD(G(\MC(A))) \rar{\rho} & G(\MC^2(A))
\end{tikzcd}
\end{equation*}

Then, the dual situation of comonads behaves dually to that of monads. In particular, theorem~\ref{thm:universal-bimorphisms} dualises as:
\begin{theorem}
    For comonads $\MC,\MD$ and a functor $G$ as above, a $\MC$-coalgebra $(A,\alpha)$ and a morphism $\rho : \MD(G(A)) \to G(\MC(A))$, assume that the following diagram is an equalizer in $\coemc{\MD}{\CD}$:
    \begin{equation}
        \label{eq:colifting-coeq-pair}
    \begin{tikzcd}
        (W_\alpha,\omega_\alpha)
        \rar
        &
        \emfr{\MD}(G(A))
        \ar[yshift=+0.5em]{rr}{\emfr{\MD}(\lambda)\circ \delta^\MT}
        \ar[swap,yshift=-0.5em]{rr}{\emfr{\MD}(G(\alpha))}
        &
        &
        \emfr{\MD}(G(\MC(A)))
    \end{tikzcd}
    \end{equation}
    Then, there exists a $\rho$-morphism $u\colon W_\alpha \to G(A)$ from $\omega_\alpha$ to $\alpha$ such that,
    for every $\rho$-morphism $h : B \to G(A)$ from a $\MD$-coalgebra $\beta$ to $\alpha$, there exists a coalgebra morphism $\widehat h : \beta \to \omega_\alpha$ such that $h = u \circ \widehat h$.
\end{theorem}

And formally dualizing of theorems~\ref{thm:functorial-universal-bimorphisms} and~\ref{thm:free-construction} into a single theorem for conciseness, gives us:
\begin{theorem}
\label{thm:colifting-and-cofree}
    If $\rho : \MD(G(A)) \nt G(\MC(A))$ is a natural transformation, and all coequalizers of the form~\eqref{eq:colifting-coeq-pair} exist, then there exists a functor $\emklf{G}{\rho} : \coemc{\MC}{\CC} \rightarrow \coemc{\MD}{\CD}$.

    If, furthermore, $\lambda$ is a Kleisli law, then there is a natural isomorphism:
    \begin{equation}
        \eilmofree{\MD}(G(A)) \cong \emklf{G}{\rho}(\eilmofree{\MC}(A))
    \end{equation}
\end{theorem}

\section{Lifting Adjunctions via Bimorphisms}
We now prove a well-known adjoint lifting theorem~\cite[Theorem 2]{johnstone1975adjoint} also given in dual form in~\cite{keigher1975adjunctions}. 
Our aim is to make explicit the application of the theory of bimorphisms in the proof.

Firstly, note that by proposition~\ref{prop:em-kleisli-facts}, item~\ref{en:em-transpose}, if we have an adjunction~$L \dashv R$, then the transpose operation:
\begin{equation}
\label{eq:eilenberg-moore-kleisli-correspondence}
\begin{tikzpicture}[stringdiagram, scale=0.5]
\path coordinate[dot, label=right:$\rho$] (rho)
+(0,1.5) coordinate[label=above:$\MS$] (tl)
+(-1,-1.5) coordinate[label=below:$R$] (bl)
+(1,1.5) coordinate[label=above:$R$] (tr)
+(0,-1.5) coordinate[label=below:$\MT$] (br);
\begin{scope}[on background layer]
\fill[catd] (bl) -- (tr) -- ++(1,0) -- ++(0,-3) -- cycle;
\fill[catc] (tr) -- (bl) -- ++(-1,0) -- ++(0,3) -- cycle;
\end{scope}
\draw (tl) -- (br) (tr) -- (bl);
\end{tikzpicture}
\mapsto
\begin{tikzpicture}[stringdiagram, scale=0.5]
\path coordinate[dot, label=right:$\rho$] (rho)
+(-2,1.5) coordinate[label=above:$L$] (tl)
+(0,1.5) coordinate[label=above:$\MS$] (tr)
+(0,-1.5) coordinate[label=below:$\MT$] (bl)
+(2,-1.5) coordinate[label=below:$L$] (br)
+(-1,-1) coordinate[dot, label=above:$\epsilon$] (epsilon)
+(1,1) coordinate[dot,label=below:$\eta$] (eta);
\draw (tr) -- (bl) (tl) to[out=-90, in=180] (epsilon.center) to[out=0, in=-135] (rho.center) to[out=45,in=180] (eta.center) to[out=0, in=90] (br);
\begin{scope}[on background layer]
\fill[catc] (tl) to[out=-90, in=180] (epsilon.center) to[out=0, in=-135] (rho.center) to[out=45,in=180] (eta.center) to[out=0, in=90] (br) -- ++(1,0) -- ++(0,3) -- cycle;
\fill[catd] (tl) to[out=-90, in=180] (epsilon.center) to[out=0, in=-135] (rho.center) to[out=45,in=180] (eta.center) to[out=0, in=90] (br) -- ++(-5,0) -- ++(0,3) -- cycle;
\end{scope}
\end{tikzpicture}
\end{equation}
yields a bijection between natural transformations of type~$\MS \circ R \nt R \circ \MT$ and those of type~$L \circ \MS \nt \MT \circ L$. Furthermore, it is easy to check this restricts to a bijection between Eilenberg-Moore laws and Kleisli laws of the corresponding types.

Assuming an~$L \dashv R$, and an Eilenberg-Moore law~$\rho : \MS \circ R \nt R \circ \MT$, let~$\lambda : L \circ \MS \nt \MT \circ L$ be defined by the mapping~\eqref{eq:eilenberg-moore-kleisli-correspondence}. We then observe, writing~$\eta$ and~$\epsilon$ for the unit and counit of the adjunction, that:
\begin{equation*}
    \label{eq:morphism-bimorphism-correspondence}
    \begin{tikzpicture}[stringdiagram, scale=0.5]
    \path coordinate[dot, label=right:$\alpha$] (alpha)
    +(0,1.5) coordinate[label=above:$A$] (tr)
    +(-2,1.5) coordinate[label=above:$\MS$] (tl)
    ++(0,-1) coordinate[dot, label=right:$h$] (h)
    +(0,-1.5) coordinate[label=below:$B$] (br)
    +(-2,-1.5) coordinate[label=below:$R$] (bl);
    \draw 
    (tr) -- (br) 
    (tl) to[out=-90, in=180] (alpha.center)
    (bl) to[out=90, in=180] (h.center);
    \begin{scope}[on background layer]
    \fill[catterm] (br) rectangle($(tr) + (1,0)$);
    \fill[catc] (br) rectangle ($(tl) + (-1,0)$);
    \fill[catd] (bl) to[out=90, in=180] (h.center) -- (br) -- cycle;
    \end{scope}
    \end{tikzpicture}
    \;=\;
    \begin{tikzpicture}[stringdiagram, scale=0.5]
    \path coordinate[dot, label=right:$h$] (h)
    +(0,1) coordinate[label=above:$A$] (tr)
    +(-2,1) coordinate[label=above:$\MS$] (tl)
    ++(0,-2) coordinate[dot, label=right:$\beta$] (beta)
    +(0,-1) coordinate[label=below:$B$] (br)
    +(-2,-1) coordinate[label=below:$R$] (bl);
    \draw (tr) -- (br);
    \draw[name path=downcurv] (h.center) to[out=180, in=90] (bl);
    \draw[name path=upcurv] (beta) to[out=180, in=-90] (tl);
    \path[name intersections={of=downcurv and upcurv}]
    coordinate[dot, label=180:$\rho$] (rho) at (intersection-1);
    \begin{scope}[on background layer]
    \fill[catterm] (br) rectangle($(tr) + (1,0)$);
    \fill[catc] (br) rectangle ($(tl) + (-1,0)$);
    \fill[catd] (h.center) to[out=180, in=90] (bl) -- (br) -- cycle;
    \end{scope}
    \end{tikzpicture}
    \quad\Leftrightarrow\quad
    \begin{tikzpicture}[stringdiagram, scale=0.5]
    \path coordinate[dot, label=right:$\alpha$] (alpha)
    +(0,1.5) coordinate[label=above:$A$] (tr)
    +(-2,1.5) coordinate[label=above:$\MS$] (tl)
    +(-3,1.5) coordinate[label=above:$L$] (tll)
    ++(0,-1) coordinate[dot, label=right:$h$] (h)
    +(-1,-1) coordinate[dot, label=above:$\epsilon$] (epsilon)
    +(0,-1.5) coordinate[label=below:$B$] (br);
    \draw 
    (tr) -- (br) 
    (tl) to[out=-90, in=180] (alpha.center)
    (tll) to[out=-90, in=180] (epsilon.center) to[out=0, in=-135] (h.center);
    \begin{scope}[on background layer]
    \fill[catterm] (br) rectangle($(tr) + (1,0)$);
    \fill[catd] (br) rectangle ($(tll) + (-1,0)$);
    \fill[catc] (tll) to[out=-90, in=180] (epsilon.center) to[out=0, in=-135] (h.center) -- (tr) -- cycle;
    \end{scope}
    \end{tikzpicture}
    \;=\;
    \begin{tikzpicture}[stringdiagram, scale=0.5]
    \path coordinate[dot, label=right:$h$] (h)
    +(-1,-1) coordinate[dot, label=above:$\epsilon$] (epsilon')
    +(-2,0) coordinate[dot, label=below:$\eta$] (eta)
    +(0,1) coordinate[label=above:$A$] (tr)
    +(-4,1) coordinate[label=above:$\MS$] (tl)
    +(-5,1) coordinate[label=above:$L$] (tll)
    ++(0,-2) coordinate[dot, label=right:$\beta$] (beta)
    +(-3,0) coordinate[dot, label=below:$\epsilon$] (epsilon)
    +(0,-1) coordinate[label=below:$B$] (br)
    +(-2,-1) coordinate[label=below:$R$] (bl);
    \draw (tr) -- (br);
    \draw[name path=downcurv] (h.center) to[out=180, in=0] (epsilon'.center) to[out=180, in=0] (eta.center) to[out=180, in=0] (epsilon.center) to[out=180, in=-90] (tll);
    \draw[name path=upcurv] (beta) to[out=180, in=-90] (tl);
    \path[name intersections={of=downcurv and upcurv}]
    coordinate[dot, label=180:$\rho$] (rho) at (intersection-1);
    \begin{scope}[on background layer]
    \fill[catterm] (br) rectangle($(tr) + (1,0)$);
    \fill[catd] (br) rectangle ($(tll) + (-1,0)$);
    \fill[catc] (h.center) to[out=180, in=0] (epsilon'.center) to[out=180, in=0] (eta.center) to[out=180, in=0] (epsilon.center) to[out=180, in=-90] (tll) -- (tr) -- cycle;
    \end{scope}
    \end{tikzpicture}
\end{equation*}
By the snake equation for the unit an counit of the adjunction. This establishes that~$h : A \to R(B)$ is a~right $\rho$-morphism if and only if its transpose $L(h)\circ \epsilon$ of type $L(A) \to B$ is a~left $\lambda$-morphism. Therefore we have a bijection:
\begin{prooftree}
    \AxiomC{$\alpha \rightarrow \emf{R}{\rho}(\beta)$}
\doubleLine
\UnaryInfC{$\alpha \klar{\lambda} \beta$}
\end{prooftree}
where~$\emf{R}{\rho}$ denotes the functor induced by the Eilenberg-Moore law~$\rho$.

If~$\emc{\MT}{\CD}$ has coequalizers of reflexive pairs, then the universal property of the classifying
objects gives a bijection:
\begin{prooftree}
\AxiomC{$\alpha \klar{\lambda} \beta$}
\doubleLine
    \UnaryInfC{$\emklf{L}{\lambda}(\alpha) \rightarrow \beta$}
\end{prooftree}
Combining these two bijections, we have show that~$\emklf{L}{\lambda} \dashv \emf{R}{\rho}$.

\begin{example}[Adjoints to functors induced by monad morphisms]
    It is well-known that a monad morphism~$\sigma : \MS \nt \MT$ can be viewed as an Eilenberg-Moore law~$\MS \circ \Id \nt \Id \circ \MT$, and trivially~$\Id_{\CC} \dashv \Id_{\CC}$. Therefore, we can apply the previous result if~$\emc{\MT}{\CC}$ has coequalizers of reflexive pairs. The Eilenberg-Moore law induces a functor~$\emf{\Id}{\sigma} : \emc{\MT}{\CC} \rightarrow \emc{\MS}{\CC}$, which has a left adjoint~$\emklf{\Id}{\sigma} : \emc{\MS}{\CC} \rightarrow \emc{\MT}{\CC}$.
\end{example}

We also note that if we have an invertible Eilenberg-Moore~$\lambda^{-1} : \MT \circ L \nt L \circ \MS$, then by proposition~\ref{prop:em-kleisli-facts}, item~\ref{en:em-inverse}, $\lambda : L \circ \MS \nt \MT \circ L$ is a Kleisli law,
and by item~\ref{en:em-transpose} its transpose~$\rho : \MS \circ R \nt R \circ \MT$ is an Eilenberg-Moore law. We can then
apply the previous result to lift this adjunction to the Eilenberg-Moore categories. This is~\cite[Theorem 4]{johnstone1975adjoint}.

\bibliographystyle{alpha}
\bibliography{bimorphisms}

\begin{thebibliography}{ADW17}

\bibitem[ADW17]{AbramskyDW17}
Samson Abramsky, Anuj Dawar, and Pengming Wang.
\newblock The pebbling comonad in finite model theory.
\newblock In {\em 32nd Annual {ACM/IEEE} Symposium on Logic in Computer
  Science, {LICS} 2017, Reykjavik, Iceland, June 20-23, 2017}, pages 1--12.
  {IEEE} Computer Society, 2017.

\bibitem[Alu21]{aluffi2021algebra}
Paolo Aluffi.
\newblock {\em Algebra: chapter 0}, volume 104.
\newblock American Mathematical Soc., 2021.

\bibitem[App65]{appelgate1965acyclic}
Harry~Wesley Appelgate.
\newblock {\em Acyclic models and resolvent functors}.
\newblock PhD thesis, Columbia University, 1965.

\bibitem[AS21]{abramsky2021relating}
Samson Abramsky and Nihil Shah.
\newblock Relating structure and power: Comonadic semantics for computational
  resources.
\newblock {\em Journal of Logic and Computation}, 31(6):1390--1428, 2021.

\bibitem[Bec69]{beck1969distributive}
Jon Beck.
\newblock Distributive laws.
\newblock In {\em Seminar on triples and categorical homology theory},
  volume~80 of {\em Lecture Notes in Mathematics}, pages 119--140. Springer,
  1969.

\bibitem[Bor94]{borceux1994handbook}
Francis Borceux.
\newblock {\em Handbook of Categorical Algebra: Volume 2, Categories and
  Structures}, volume~2.
\newblock Cambridge University Press, 1994.

\bibitem[BW00]{barr2000toposes}
Michael Barr and Charles Wells.
\newblock {\em Toposes, triples, and theories}.
\newblock Reprints in Theory and Applications of Categories, 2000.
\newblock Originally published by Springer-Verlag.

\bibitem[CE12]{courcelle2012graph}
Bruno Courcelle and Joost Engelfriet.
\newblock {\em Graph structure and monadic second-order logic: a
  language-theoretic approach}, volume 138.
\newblock Cambridge University Press, 2012.

\bibitem[CMR00]{courcelle2000linear}
Bruno Courcelle, Johann~A Makowsky, and Udi Rotics.
\newblock Linear time solvable optimization problems on graphs of bounded
  clique-width.
\newblock {\em Theory of Computing Systems}, 33(2):125--150, 2000.

\bibitem[Cou90]{courcelle1990monadic}
Bruno Courcelle.
\newblock The monadic second-order logic of graphs. {I}. {R}ecognizable sets of
  finite graphs.
\newblock {\em Information and computation}, 85(1):12--75, 1990.

\bibitem[FV67]{feferman1967first}
Solomon Feferman and Robert~L. Vaught.
\newblock The first order properties of products of algebraic systems.
\newblock {\em Journal of Symbolic Logic}, 32(2):57--103, 1967.

\bibitem[Gur85]{gurevich1985monadic}
Yuri Gurevich.
\newblock Monadic second-order theories.
\newblock In J.~Barwise and S.~Feferman, editors, {\em Model-theoretic logics},
  pages 479--506. Cambridge University Press, 1985.

\bibitem[HM16a]{HinzeM16b}
Ralf Hinze and Dan Marsden.
\newblock Dragging proofs out of pictures.
\newblock In Sam Lindley, Conor McBride, Philip~W. Trinder, and Donald
  Sannella, editors, {\em A List of Successes That Can Change the World -
  Essays Dedicated to Philip Wadler on the Occasion of His 60th Birthday},
  volume 9600 of {\em Lecture Notes in Computer Science}, pages 152--168.
  Springer, 2016.

\bibitem[HM16b]{HinzeM16a}
Ralf Hinze and Dan Marsden.
\newblock Equational reasoning with lollipops, forks, cups, caps, snakes, and
  speedometers.
\newblock {\em J. Log. Algebraic Methods Program.}, 85(5):931--951, 2016.

\bibitem[Jac94]{jacobs1994semantics}
Bart Jacobs.
\newblock Semantics of weakening and contraction.
\newblock {\em Annals of pure and applied logic}, 69(1):73--106, 1994.

\bibitem[JMS22]{jaklmarsdenshah2022fvm}
Tom\'{a}\v{s} Jakl, Dan Marsden, and Nihil Shah.
\newblock A game comonadic account of {C}ourcelle and
  {F}eferman-{V}aught-{M}ostowski theorems.
\newblock In preparation, 2022.

\bibitem[Joh75]{johnstone1975adjoint}
Peter~T Johnstone.
\newblock Adjoint lifting theorems for categories of algebras.
\newblock {\em Bulletin of the London Mathematical Society}, 7(3):294--297,
  1975.

\bibitem[Kei75]{keigher1975adjunctions}
William Keigher.
\newblock Adjunctions and comonads in differential algebra.
\newblock {\em Pacific journal of Mathematics}, 59(1):99--112, 1975.

\bibitem[Koc70]{Kock1970}
Anders Kock.
\newblock Monads on symmetric monoidal closed categories.
\newblock {\em Archiv der Mathematik}, 21(1):1--10, 1970.

\bibitem[Koc71]{kock1971bilinearity}
Anders Kock.
\newblock Bilinearity and cartesian closed monads.
\newblock {\em Mathematica Scandinavica}, 29(2):161--174, 1971.

\bibitem[Koc72]{Kock1972}
Anders Kock.
\newblock Strong functors and monoidal monads.
\newblock {\em Archiv der Mathematik}, 23(1):113--120, 1972.

\bibitem[Mak04]{makowsky2004algorithmic}
Johann~A Makowsky.
\newblock Algorithmic uses of the {F}eferman--{V}aught theorem.
\newblock {\em Annals of Pure and Applied Logic}, 126(1-3):159--213, 2004.

\bibitem[ML13]{mac2013categories}
Saunders Mac~Lane.
\newblock {\em Categories for the working mathematician}, volume~5.
\newblock Springer Science \& Business Media, 2013.

\bibitem[MM07]{manes2007monad}
Ernie Manes and Philip Mulry.
\newblock Monad compositions. i: General constructions and recursive
  distributive laws.
\newblock {\em Theory and Applications of Categories}, 18:172--208, 2007.

\bibitem[Mos52]{mostowski1952direct}
Andrzej Mostowski.
\newblock On direct products of theories.
\newblock {\em The Journal of Symbolic Logic}, 17(1):1--31, 1952.

\bibitem[OS06]{oum2006approximating}
Sang-il Oum and Paul Seymour.
\newblock Approximating clique-width and branch-width.
\newblock {\em Journal of Combinatorial Theory, Series B}, 96(4):514--528,
  2006.

\bibitem[PT04]{pedicchio2004categorical}
Maria~Cristina Pedicchio and Walter Tholen.
\newblock {\em Categorical Foundations: Special Topics in Order, Topology,
  Algebra, and Sheaf Theory}.
\newblock Cambridge University Press, 2004.

\bibitem[Sea13]{seal2012tensors}
Gavin~J Seal.
\newblock Tensors, monads and actions.
\newblock {\em Theory and Applications of Categories}, 28:403--434, 2013.

\end{thebibliography}



\end{document}